\documentclass{gtpart}     
\usepackage{latexsym}
\usepackage{amsfonts}
\usepackage{amsbsy}
\usepackage{mathrsfs}
\usepackage{epsfig}
\usepackage[all]{xy}
\usepackage{url}
\usepackage{tikz}
\usetikzlibrary{arrows,matrix}
\usepackage{hyperref}

\title{Rectification of Weak Product Algebras over an Operad in $\Cat$ and $\Top$ and Applications}

\author[Fiedorowicz]{Z. Fiedorowicz}
\givenname{Zbigniew}
\surname{Fiedorowicz}
\address{Department of Mathematics, The Ohio State University\\ Columbus, OH 43210-1174, USA}
\email{fiedorow@math.ohio-state.edu}
\urladdr{http://www.math.ohio-state.edu/people/fiedorow/view}

\author[Stelzer]{M. Stelzer}
\givenname{Manfred}
\surname{Stelzer}
\address{Universit\"at Osnabr\"uck, Fachbereich Mathematik/Informatik\\ Albrechtstr. 28a, 49069 Osnabr\"uck, Germany}
\email{mstelzer@mathematik.uni-osnabrueck.de}

\author[Vogt]{R.M.~Vogt}
\givenname{Rainer}
\surname{Vogt}
\address{Universit\"at Osnabr\"uck, Fachbereich Mathematik/Informatik\\ Albrechtstr. 28a, 49069 Osnabr\"uck, Germany}
\email{rainer@mathematik.uni-osnabrueck.de}
\urladdr{http://www.mathematik.uni-osnabrueck.de/staff/phpages/vogtr.rdf.shtml}

\keyword{Categories, Operads}
\subject{primary}{msc2000}{18D50}
\subject{secondary}{msc2000}{55P48}

\volumenumber{}
\issuenumber{}
\publicationyear{}
\papernumber{}
\startpage{}
\endpage{}
\doi{}
\MR{}
\Zbl{}
\received{}
\revised{}
\accepted{}
\published{}
\publishedonline{}
\proposed{}
\seconded{}
\corresponding{}
\editor{}
\version{}

\makeatletter
\@addtoreset{equation}{section}

\makeatother
\makeatletter
\@addtoreset{equation}{section}

\makeatother

\newtheorem{prop}{Proposition}[section]
\newtheorem{theo}[prop]{Theorem}
\newtheorem{lem}[prop]{Lemma}
\newtheorem{const}[prop]{Constructions}
\newtheorem{coro}[prop]{Corollary}

\theoremstyle{definition}
\newtheorem{defi}[prop]{Definition}

\newtheorem{leer}[prop]{}
\newtheorem{rema}[prop]{Remark}

\def\scO{\mathcal{O}}
\def\hscO{\widehat{\mathcal{O}}}
\def\scP{\mathcal{P}}

\def\scC{\mathcal{C}}
\def\scD{\mathcal{D}}
\def\scM{\mathcal{M}}
\def\scI{\mathcal{T}}

\def\scA{\mathcal{A}}
\def\scB{\mathcal{B}}
\def\scD{\mathcal{D}}

\def\scF{\mathcal{F}}
\def\scL{\mathcal{L}}
\def\scK{\mathcal{K}}
\def\scS{\mathcal{S}}
\def\scT{\mathcal{T}}
\def\scU{\mathcal{U}}
\def\scV{\mathcal{V}}
\def\mT{\mathbb{T}}
\def\tmT{\widetilde{\mathbb{T}}}
\def\mN{\mathbb{N}}

\def\hocolim{\textrm{hocolim}}

\def\Id{{\textrm{Id}}}
\def\id{{\textrm{id}}}
\def\op{{\textrm{op}}}
\def\ob{{\textrm{ob}}}
\def\proj{\textrm{proj}}
\def\product{\textrm{product}}
\def\Cat{\mathcal{C}\!at}
\def\Top{\mathcal{T}\!op}

\def\SSets{\mathcal{SS}ets}
\def\Sets{\mathcal{S}\!ets}
\def\Inj{\textrm{Inj}}
\def\In{\textrm{In}}
\def\we{\textrm{we}}

\def\const{\textrm{const}}
\def\top{\textrm{top}}
\def\cat{\textrm{cat}}
\def\Aut{\mathop{\rm Aut}}
\def\Coll{\mathop{\mathcal{C}oll}}
\def\Opr{\mathop{\mathcal{O}pr}}

\def\lambdaov{\overline{\lambda}}

\begin{document}
\begin{abstract}    
We develop an alternative to the May-Thomason construction
used to compare operad based infinite loop machines to that of Segal,
which relies on weak products. Our construction has the advantage that
it can be carried out in $Cat$, whereas their construction gives rise to
simplicial categories. As an application we show that a simplicial algebra
over a $\Sigma$-free $Cat$ operad $\scO$ is functorially weakly equivalent to a
$Cat$ algebra over $\scO$. When combined with the
results of a previous paper, this allows us to conclude that 
up to weak equivalences the category of $\scO$-categories is equivalent to the category
of $B\scO$-spaces, where $B:\Cat\to\Top$ is the classifying space functor. 
In particular,
$n$-fold loop
spaces (and more generally $E_n$ spaces) are functorially weakly equivalent
to classifying spaces of $n$-fold monoidal categories.  Another application
is a change of operads construction within $Cat$.
\end{abstract}

\maketitle

\vspace{2ex}
\section{Introduction}

In \cite{MT} May and Thomason compared infinite loop machines based on spaces with an operad
acting on them and the Segal machine which involves weakening the notion of Cartesian product
to that of a product up to equivalence.  In the process they introduced a hybrid notion of an
algebra over a category of operators and created a rectification construction to pass from this
to an equivalent space with an operad action.  Their rectification is a 2-sided bar construction,
which is simplicial in nature.
Schw\"anzl and Vogt gave an alternative comparison of the two infinite  loop space machines in \cite{SV},
 which is based on the fact that for a strong deformation retract $A\subset X$ the space of strong deformation 
retractions $X\to X$ is contractible. Both approaches do not directly translate to $\Cat$, the category 
of small categories, with realization equivalences as weak equivalences: The May-Thomason construction
would convert categories into simplicial categories, and there is no apparent candidate to replace 
the space of strong deformation retractions in the Schw\"anzl-Vogt construction.

Similarly, the change of operads construction used in \cite{May1}, if applied to operads in $\Cat$,
ends up in simplicial categories.

In this paper we offer a comparatively simple third rectification which has the advantages that 
it can be carried out in $Cat$ and that a change of operads functor based on it
 stays in $\Cat$.

Our main motivation for this paper is to realize a program started in \cite{BFSV}, where a
notion of $n$-fold monoidal category was introduced whose structure is codified by a $\Sigma$-free
operad $\scM_n$ in $\Cat$. The classifying space functor $B:\Cat\to \Top$ maps $\scM_n$ to a
topological operad $B\scM_n$, and it was shown in \cite{BFSV} that there is a topological operad $\scD$
and equivalences of operads
$$B\scM_n\leftarrow \scD \rightarrow \scC_n$$
where $\scC_n$ is the little $n$-cubes operad. A change of operads construction for topological operads
then implies that the classifying space $B\scA$ of any $n$-fold monoidal category $\scA$ is weakly equivalent to a 
$\scC_n$-space and hence to an n-fold loop space up to group completion.
It was conjectured that any $n$-fold loop space can be obtained up to equivalence in this way. 

More generally, let $\scO$ and  $\scP$ be $\Sigma$-free operads in $\Cat$ respectively $\Top$,
and let $\scO\mbox{-}\Cat$ and $\scP\mbox{-}\Top$ be their associated categories of algebras.
Taking $\mathcal{P}=B\mathcal{O}$, one might be tempted to conjecture that
the classifying space functor induces an equivalence of categories
$$
\scO\mbox{-}\Cat[\we^{-1}]\simeq B\scO\mbox{-}\Top[\we^{-1}]
$$
where $\we\subset B\scO\mbox{-}\Top$ is the class of all homomorphisms whose underlying maps are weak
homotopy equivalences and $\we\subset \scO\mbox{-}\Cat$ is the class of all homomorphisms
which are mapped to weak equivalences in $B\scO\mbox{-}\Top$. To ensure the existence of the localized
categories $B\scO\mbox{-}\Top[\we^{-1}]$ and $\scO\mbox{-}\Cat[\we^{-1}]$ we can use Grothendieck's language
of universes \cite[Appendix]{Grot}, where they exist in some higher universe.

A partial step towards a proof
was accomplished in \cite{FV}, where it was shown that the classifying space functor
followed by the topological realization functor induces an equivalence of categories
$$
\scO\mbox{-}\scS\Cat[\we^{-1}]\simeq B\scO\mbox{-}\Top[\we^{-1}]
$$ 
where $\scO\mbox{-}\scS\Cat$ is the category of simplicial $\scO$-algebras in $\Cat$ and the weak
equivalences in $\scO\mbox{-}\scS\Cat$ are those homomorphisms which are mapped to weak equivalences
in $B\scO\mbox{-}\Top$.
In particular, each $E_n$-space is up to
equivalence the classifying space of a simplicial $n$-fold monoidal category. As far as $E_n$-spaces
 are concerned the full program
was finally realized in \cite{FSV}, where a homotopy colimit construction for categories of algebras
over a $\Sigma$-free operad in $\Cat$ provided a passage from simplicial
$\scO$-algebras to $\scO$-algebras. If the morphisms of the operad $\scO$ satisfy a certain factorization
condition this passage induces an equivalence of categories
$$
\scO\mbox{-}\scS\Cat[\we^{-1}]\simeq \scO\mbox{-}\Cat[\we^{-1}]
$$
and the operads codifying $n$-fold monoidal categories, strictly associative braided monoidal categories, and permutative
categories satisfy this condition. For these operads it was also shown that there is an equivalence of categories
$$(*)\qquad\qquad\qquad\qquad\qquad
\scO\mbox{-}\Cat\widetilde{[\we^{-1}]}\simeq B\scO\mbox{-}\Top[\we^{-1}]
\qquad\qquad\qquad\qquad\qquad$$
in the foundational setting of G\"odel-Bernays, where $\scO\mbox{-}\Cat\widetilde{[\we^{-1}]}$ is a localization of 
$\scO\mbox{-}\Cat$ up to equivalence (for a definition see \cite[Def. 7.3]{FSV}).

The main application of the construction developed in this paper is the full proof of the above conjecture in the foundational
setting of G\"odel-Bernays with no restrictions on the operad $\scO$ in $\Cat$ apart from $\Sigma$-freeness. 
For the existence of the genuine localizations we utilize an observation
of Schlichtkrull and Solberg \cite[Prop. A.1]{SS}, and we thank them for communicating this to us.
As far as $E_n$-spaces are concerned, the present paper offers an alternative simpler proof, because it avoids the 
comparatively complicated homotopy colimit construction in $\scO\mbox{-}\Cat$, which is of independent interest. In particular,
it considerably simplifies the part of the proof of the main result of Thomason in \cite{Thom2} (the special case of $(*)$ 
for the operad encoding permutative categories), which relies on the homotopy colimit construction of \cite{Thom1}.

The genesis of this paper stems from a previous paper of two of the authors, \cite[\S 4]{FGV}, where a similar
problem involved the rectification of a weak monoidal structure on a category, without passing to simplicial
categories. It was observed there that the classical $M$-construction of \cite[Theorem 1.26]{BV}, used for
this kind of rectification in $\Top$, could be carried out in $\Cat$.  This led us to seek a modification of this
construction for the purpose of rectifying weak product algebras in $\Cat$.

This paper is organized as follows: In Section 2 we recall some basic notions of operads and their associated
categories of operators.  In Section 3 we recall free operad constructions and the language of trees, which
underlie our rectification constructions.  In Section 4 we construct a modification of the $M$-construction in
$\Top$ which allows weak product algebras over an operad as inputs.  In Sections 5 and 6 we recast our
modified $M$-construction as a homotopy colimit of a diagram in $\Top$.  Building upon work of Thomason \cite{T},
we then show that the Grothendieck construction on the same diagram in $\Cat$ provides the requisite rectification
of weak product algebras over an operad in $\Cat$.  The remaining sections are then devoted to various applications
of our rectification construction.

We would like to thank the referee for a careful reading of this paper and some helpful suggestions.  The first author
also wishes to acknowledge the support of the University of Osnabr\"uck in the preparation of this paper.
\section{Operads and their categories of operators}\label{operads}
For the reader's convenience we recall the notions of an operad and its associated category of operators.

Let $\scS$ be either the category $\Cat$ of small categories, or the category $\Sets$ of sets,
or the category $\SSets$ of simplicial sets, or the category $\Top$ of 
(not necessarily Hausdorff) $k$-spaces. Then $\scS$ is a self-enriched symmetric monoidal category with the product as structure
functor and the terminal object $\ast$ as unit.  In what follows, for an object $X$ in $\scS$, it will be convenient to refer to \textit{elements}
in $X$.  If $X$ is a topological space, this will mean a point in $X$.  If $X$ is a simplicial set, this will mean a simplex in $X$.  If
$X$ is a category, then this will mean either an object or morphism in $X$. We will also use the following notions of equivalence
in $\scS$.  In $\Top$ an equivalence will mean a strict homotopy equivalence.  An equivalence between simplicial sets will mean
a simplicial map whose geometric realization is a homotopy equivalence.  Lastly in $\Cat$, we will call a functor $F:\scC\to\scD$ an equivalence if it induces a homotopy equivalence on the geometric realizations of the nerves.

\begin{defi}\label{operads1}
An \textit{operad} $\scO$ in $\scS$ is a collection $\{\scO (k)\}_{k\ge 0}$ of
objects in $\scS$ equipped with symmetric group actions
$\scO (k)\times \Sigma_k\to\scO (k)$, composition maps
$$
\scO (k)\times(\scO (j_1)\times\ldots\times \scO (j_k))
\to \scO (j_1+\ldots+j_k),
$$
and a unit $\id\in\scO(1)$ satisfying the appropriate equivariance,
associativity and unitality conditions - see \cite{May1} for details.\\
An operad in $\Top$ is called well-pointed if $\{\id\}\subset \scO(1)$ is
a closed cofibration.
\end{defi}

\begin{rema}\label{operads2}
We often find it
helpful to think of an operad in the following equivalent way. An operad $\mathcal{O}$ in 
$\scS$ is an $\scS$-enriched symmetric monoidal category $(\scO,\oplus,0)$ such that 
\begin{enumerate}
\item[(i)] $ob\:\scO=\mathbb{N}$ and $m\oplus n=m+n$.
\item[(ii)] $\oplus$ is a strictly associative $\scS$-functor with strict unit 0
\item[(iii)] $\coprod\limits_{r_1+\ldots+r_n=k}\mathcal{O}(r_1,1)
\times\ldots\times \mathcal{O}(r_n,1)\times_{\Sigma_{r_1}
\times\ldots\times\Sigma_{r_n}} \Sigma_k\to\mathcal{O}(k,n)$
$$
\xymatrix{
((f_1,\ldots,f_n),\tau) \ar@{|->}[rr] 
&& (f_1\oplus\ldots\oplus f_n)\circ\tau
}$$
is an isomorphism in $\scS$. (Note in particular that $\scO(n,0)=\emptyset$ for $n>0$.  By contrast, there are no a priori
restrictions on $\scO(0,1)$.)
\end{enumerate}
In the topological case ``well-pointed'' translates to $\{ \id \}\subset \mathcal{O}(1,1)$
is a closed cofibration.

Each such category determines an operad in the sense of Definition \ref{operads1} by taking the 
collection  $\{\mathcal{O}(k,1)\}_{k\ge 0}$.  Conversely, each operad determines such a category
by property (iii).

The symmetric monoidal category associated to the trivial operad $\scC om$, which parametrizes
commutative monoid structures, can be identified with a skeletal category of unbased finite sets $\scF$.
Here we identify the natural number $n$ with the set $\{1,2,\dots,n\}$, which may be viewed as an object
in any of our categories $\scS$.  In particular we identify $0$ with the empty set.
For any operad $\scO$, the natural map $\scO\to \scC om$ induces a symmetric monoidal functor
$\epsilon:\scO\to\scF$. This functor induces an equivalence on $\scS$-enriched morphism sets
$\amalg_{m,n}\scO(m,n)\to\amalg_{m,n}\scF(m,n)$ for any $E_\infty$ operad $\scO$.  More generally
for any morphism $\phi: m\to n$ in $\scF$ and any operad $\scO$, $\epsilon^{-1}(\phi)$ is isomorphic
to the product $\prod_{i=1}^n\scO(|\phi^{-1}(i)|)$, where $|S|$ denotes the cardinality of the set $S$.
\end{rema}

\begin{defi}\label{operads3}
 Let $\mathcal{O}$ and $\mathcal{P}$ be  operads in $\scS$.
\begin{enumerate}
\item[(1)] $\mathcal{O}$ is called $\Sigma$\textit{-free} if the $\Sigma_n$-action on $\mathcal{O}(n)$
is free for each $n$ in the cases $\scS=\Cat$, $\Sets$, or $\SSets$. If $\scS=\Top$ we require that
$\mathcal{O}(n)\to\mathcal{O}(n)/\Sigma_n$ is a numerable principal
$\Sigma_n$-bundle for each $n$.
\item[(2)] An \textit{operad map} $\mathcal{O}\to\mathcal{P}$ is a collection of
equivariant maps $\mathcal{O}(n)\to\mathcal{P}(n)$ in $\scS$, compatible with the
operad structure.
\item[(3)] An $\mathcal{O}$\textit{-structure} on an object $X$ in $\scS$ is an operad map
$\mathcal{O}\to\mathcal{E}_X$ into the endomorphism operad $\mathcal{E}_X$
of $X$, which is defined by $\mathcal{E}_X(n)=\scS(X^n,X)$ with the obvious $\Sigma_n$-action and the
obvious composition maps and unit. We say that $\mathcal{O}$ \textit{acts on} $X$, or that $X$ is an
$\mathcal{O}$-algebra; if $\scS=\Top$ we also call $X$ an $\mathcal{O}$-space. \\
If we interpret an operad as a symmetric monoidal category as in Remark \ref{operads2} an $\mathcal{O}$-algebra
is the same as a strict symmetric monoidal $\scS$-functor $\scO\to\scS$ taking $n$ to $X^n$.  Here strict
monoidal means that we use the canonical isomorphisms in $\scS$ to identify $X^{m+n}$ with $X^m\times X^n$.
\item[(4)] An operad map is called a \textit{weak equivalence} if each map
$\mathcal{O}(n)\to\mathcal{P}(n)$ is an equivariant homotopy equivalence (in $\Cat$ or $\SSets$ this means that
each map is an equivariant homotopy equivalence after applying the classifying space functor, respectively the topological
realization functor). 
\item[(5)] Two operads are called \textit{equivalent} if there is a finite chain of weak
equivalences connecting them.
\end{enumerate}
\end{defi}
We denote the category of $\scO$-algebras in $\scS$ by $\scO\mbox{-}\scS$.

\begin{leer}\label{operads4}
Let $\scO$ be an operad in $\scS$, interpreted as in Remark \ref{operads2}. As is shown in \cite[Chapter II]{BV},
the symmetric monoidal category $\scO$ can be enlarged into
an $\scS$-enriched category with products $\Theta_\scO$, such that
$n=1\times 1\times\dots\times 1$.  This category $\Theta_\scO$ is called the \textit{theory} associated to
$\scO$ and is determined up to isomorphism by the requirement that an $\scO$-structure on an object $X$ extend uniquely to a product preserving functor $\widetilde{X}:\Theta_\scO\to\scS$.  The category
$\Theta_\scO$ contains $\scO$ and $\Pi$, the category of projections, as subcategories, and $\scO\cap\Pi=\Sigma_*$, the subcategory of bijections.
We define the \textit{category of operators} $\hscO$ as the subcategory of $\Theta_\scO$ generated by $\scO$
and $\Pi$, and note that the symmetric monoidal structure on $\Theta_\scO$ restricts to $\hscO$.
For $X$ an $\scO$-algebra, the functor $\widetilde{X}:\Theta_\scO\to\scS$ restricts to a strict symmetric
monoidal functor $\widehat{X}:\hscO\to\scC$.

A more explicit description of $\hscO$ can be obtained as follows.  First observe that for any set $S$, a
projection $S^l\to S^k$, corresponds to an injection $k\to l$ of finite sets.  Thus the category of projections $\Pi$
can be identified with $\mbox{Inj}^{op}$, the opposite of the category of injections in $\scF$.
Then
\[\hscO(l,n)=\coprod\limits_{0\leq k\leq l} \scO(k,n)\times_{\Sigma_k}\Inj (k,l).\]
In particular $\hscO(l,0)$ consists of a single morphism, the nullary projection.
Composition of $(f,\sigma)\in \scO(k,n)\times_{\Sigma_k} \Inj(k,l)$ with 
$(g_1\oplus\cdots\oplus g_l,\tau)\in \scO(p,l)\times_{\Sigma_p} \Inj(p,q)$, where $g_i\in \scO(r_i,1)$ and $p=r_1 +\cdots + r_l$,
defined by
\[ (f,\sigma)\circ (g_1\oplus\cdots\oplus g_l,\tau) =
(f\circ (g_{\sigma(1)}\oplus\cdots\oplus g_{\sigma(k)}), \tau\circ \sigma(r_1,\ldots ,r_l))
\]
where $\sigma(r_1,\ldots ,r_l): \underline{r}=\underline{r_{\sigma(1)}}+\cdots +\underline{r_{\sigma(k)}}
\to \underline{p}$ is the following block injection: $\underline{r}$ is the ordered disjoint union
$\underline{r}=\underline{r_{\sigma(1)}}\amalg \cdots \amalg \underline{r_{\sigma(k)}}$ and 
$\underline{p}$ is the ordered disjoint union $\underline{r_1}\amalg \cdots \amalg \underline{r_l}$; the
block injection $\sigma(r_1,\ldots ,r_l)$ maps the block $\underline{r_{\sigma(i)}}$ identically onto
the corresponding block in $\underline{p}$. For a comparison of this description of $\hscO$ with that given
in May-Thomason \cite{MT}, the reader is referred to the proof of Lemma 5.7 in that paper.
\end{leer}

We will often denote the morphisms $(\id_k,\sigma)\in \hscO(l,k)$ by $\sigma^\ast$.

\begin{rema}\label{operads4a}
 \begin{enumerate}
 \item If $\scO= \scC om$, then $\hscO$ can be identified with $\scF_*$, the skeletal category of
\textit{based} finite sets, with objects $n_+=\{0,1,2,\dots,n\}$.  The inclusion $\scO\subset\hscO$ can be identified with the functor $\scF\to\scF_*$, which adjoins a disjoint basepoint $0$ to the finite set $n=\{1,2,\dots,n\}$.
The theory $\Theta_{\scC om}$
can be identified with the category whose objects are the natural numbers with morphisms $m\to n$ being
$n\times m$ matrices with entries in the natural numbers, with composition given by multiplication of matrices.
We can then identify $\hscO\cong\scF_*$ with the subcategory of $\Theta_{\scC om}$ whose morphisms are
matrices with entries in $\{0,1\}$, with at most one non-zero entry in each column.
\item The unique map of operads $\scO\to \scC om$ induces  functors $\hat{\epsilon}:\hscO\to\scF_*$ and
$\Theta_\epsilon:\Theta_\scO\to\Theta_{\scC om}$, and there is a pullback diagram of $\scS$ enriched categories
\[\xymatrix{
\quad\hscO\quad\ar@{^{(}->}[rr]\ar[d]^{\hat{\epsilon}} &&\quad\Theta_\scO\quad\ar[d]^{\Theta_\epsilon}\\
\quad\scF_*\quad\ar@{^{(}->}[rr] &&\quad\Theta_{\scC om}\quad
}\]
For any $E_\infty$ operad $\scO$, $\hat{\epsilon}$ induces an equivalence on $\scS$-enriched morphism sets
$\amalg_{m,n}\hscO(m,n)\to \amalg_{m,n}\scF_*(m,n)$.  More generally
for any morphism $\phi: m\to n$ in $\scF_*$ and any operad $\scO$, $\epsilon^{-1}(\phi)$ is isomorphic
to the product $\prod_{i=1}^n\scO(|\phi^{-1}(i)|)$  Moreover $\scF_*$ is the largest subcategory of
$\Theta_{\scC om}$ containing $\scF$ with this property.  For other morphisms in $\Theta_{\scC om}$ the inverse image under $\Theta_\epsilon$ is the quotient of such a product by a stabilizing group of permutations.  For
instance
\[\Theta_\epsilon^{-1}\left(\left(\begin{array}{cc}a &b\\c &d\end{array}\right):2\longrightarrow 2\right)\cong
\left(\scO(a+b)/\Sigma_a\times\Sigma_b\right)\times\left(\scO(c+d)/\Sigma_c\times\Sigma_d\right).\]
 \end{enumerate}
\end{rema}

\begin{defi}\label{operads5}
An $\hscO$-diagram in $\scS$ is an $\scS$-enriched functor $G: \hscO\to \scS$. Such a diagram is
called \textit{special} if the injections $\iota_k:\underline{1} \to \underline{n}$,
sending $1$ to $k$ define a homotopy equivalence $(\iota_1^\ast,\ldots,\iota_n^\ast): G(n)\to G(1)^n$ for each $n$, i.e. $G$ is a weakly symmetric monoidal functor.\\
We denote the category of $\hscO$-diagrams in $\scS$ by $\scS^{\hscO}$.
\end{defi}

As we noted above in \ref{operads4}, there is an obvious functor
$$
\widehat{(-)}: \scO\mbox{-}\scS\to \scS^{\hscO},$$
given by extending a symmetric monoidal functor $X:\scO\to\scS$ to a product preserving functor $\widetilde{X}:\Theta_\scO\to\scS$, and then restricting this extension to $\widehat{X}:\hscO\to\scS$.  Explicitly $\widehat{X}:\hscO\to\scS$ is
defined by $\widehat{X}(n)=X^n$, $\widehat{X}((f,id))=X(f)$, and $\widehat{X}(\sigma^\ast):X^l\to X^k$ being the projection
$$(x_1,\ldots ,x_l)\mapsto (x_{\sigma(1)},\ldots ,x_{\sigma(k)}).$$ 
 By construction, $\widehat{X}$ is special. 
 
 We recall that the classifying space functor $B:\Cat\to \Top$ is the composite of the nerve functor $N_\ast$ and the
topological realization
$$
B:\Cat\xrightarrow{N_\ast}\SSets\xrightarrow{|-|}\Top.
$$
The classifying space functor preserves products, which implies

\begin{lem}\label{operads6a} Let $\scO$ be an operad in $\Cat$, let $X$ be an $\scO$-algebra, and let $G:\hscO\to \Cat$ be an
 $\hscO$-diagram. Then
 \begin{enumerate}
  \item $B\scO$ is an operad in $\Top$ and $BX$ is a $B\scO$-space.
  \item $B(\hscO) \cong \widehat{B\scO}$ and $BG: \widehat{B\scO}\to \Top$ is a $\widehat{B\scO}$-diagram. If $G$ is special, so is $BG$.
  \item $\widehat{BX}\cong B\widehat{X}$.
 \end{enumerate}
\end{lem}

If we want to determine the homotopy types of our categorical constructions,
 we usually have to assume that all operads we consider are $\Sigma$-free.  The
reason we need to make this assumption is that our constructions will require us to take
quotients, by permutation groups, of categories which are products of various
$\scO(k)$ categories together with other categories.  Under this hypotheses the
classifying spaces of the resulting quotient categories will be homeomorphic to the
quotients of the classifying spaces of the product categories, due to the fact that the classifying
space functor preserves finite products and the following elementary
result.

\begin{lem}\label{operads6} Let a discrete group $\Gamma$ act freely on a small category $\scC$. Then
$B\scC$ is a free $\Gamma$-space and
$$B(\scC/\Gamma)\cong (B\scC)/\Gamma.$$
\end{lem}

\begin{proof} Since $B\scC$ is a $\Gamma$-CW complex with a free $\Gamma$ action, $B\scC\to (B\scC)/\Gamma$ is a numerable
principal bundle. We have $ob(\scC/\Gamma)=ob(\scC)/\Gamma$ and $mor(\scC/\Gamma)=mor(\scC)/\Gamma$.  Composition in
$\scC/\Gamma$ is defined by lifting to $\scC$: given composable morphisms $[f]:[A]\to[B]$ and $[g]:[B]\to[C]$
in $\scC/\Gamma$, choose a representative object $A$ in $\scC$.  Then there are unique morphisms $f:A\to B$
and $g:B\to C$ in $\scC$ representing $[f]$ and $[g]$ and $[g][f]=[gf]$.  Hence any simplex in the nerve
of $\scC/\Gamma$ has a unique lift to the nerve of $\scC$ once we choose a lift of the initial vertex. It follows
that the nerve of $\scC/\Gamma$ is the quotient of the nerve of $\scC$ by the action of $\Gamma$, which implies
the result.
\end{proof}

This result fails to hold if the action of $\Gamma$ on $\scC$ is not free. For instance if
$H$ is a group regarded as a category with one object, $\scC=H\times H$ and $\Gamma=\mathbf{Z}/2$
acts on $\scC$ by permuting the factors, then $B(\scC/\Gamma)\cong BA$, where $A$ is the abelianization of $H$.
 This is clearly different from 
$B(H\times H)/\Gamma=(BH\times BH)/\left(\mathbf{Z}/2\right)$, particularly if $H$ is perfect.

\section{Free operads}\label{free}
Since our constructions start with free operads we recall their construction for the convenience of the reader and to fix notations. 
We follow the expositions \cite[Section 5.8]{Berger} and \cite[Section 3]{Berger2}, because they are the most convenient one for our purposes.
We recommend \cite[Part I, \S2]{MT} for background on the language of trees.

Recall that a \textit{collection} $\mathcal{K}$ in one of our categories $\mathcal{S}$ is an $\mathbb{N}$-indexed family of objects 
$\mathcal{K}(n)$ with a right $\Sigma_n$-action. Let $\Opr(\mathcal{S})$ and $\Coll(\mathcal{S})$ denote the categories of operads 
and collections in $\mathcal{S}$. Then there is the obvious forgetful functor
$$
R:\Opr(\mathcal{S})\to\Coll(\mathcal{S})
$$
and we are interested in its left adjoint
$$
L:\Coll(\mathcal{S})\to\Opr(\mathcal{S})
$$
the \textit{free operad functor}.

Let $\mathbb{T}$ denote the groupoid of planar trees and non-planar isomorphisms.
 Its objects are finite directed rooted planar trees [cf. \cite[pp. 85-87]{L} for a 
formal definition].  A tree can have three types of edges: \textit{internal edges} with a node on each end, \textit{input edges} with a node 
only at the end, and one outgoing edge, called the \textit{root}, with a node only at its beginning. Each node $\nu$ has a 
finite totally ordered poset $\underline{\In(\nu)}$ of incoming edges, also called \textit{inputs} of $\nu$, 
and exactly one outgoing edge, called its \textit{output}. 
The cardinality $\In(\nu)$ of $\underline{\In(\nu)}$ is called the \textit{valence} of $\nu$. We allow 
\textit{stumps}, i.e. nodes of valence 0, 
and the \textit{trivial tree} consisting of a single edge. The poset of input edges of a tree $T$ is 
denoted by $\underline{\In(T)}$, and its cardinality by $\In(T)$. We say that a subtree $T'$ of a tree $T$ is a \textit{subtree above a node}
$\nu$ of $T$ if $T'$ consists of an incoming edge of $\nu$ and all nodes and edges of $T$ lying above that edge. Note that if $T'$ is
such a subtree, then $\underline{\In(T')}$ forms a (possibly empty) subinterval of $\underline{\In(T)}$.

\begin{defi}\label{ind_perm}
A morphism $\phi:T\to T'$ in $\mathbb{T}$ is an isomorphism of trees after forgetting their planar structures. 
So $\phi$ preserves inputs and hence induces
a bijection $\In(\phi): \underline{\In(T)}\to \underline{\In(T')}$.
If $in_1,\ldots,in_n$ are the inputs of $T$ and
$in'_1,\ldots,in'_n$ are the inputs of $T'$ counted from left to right, then $\phi$ has an associated 
permutation $\phi^\Sigma\in \Sigma_n$ defined by $\phi^\Sigma(k)=l$ if
$\phi(in_l)=in'_k$.  Note that $\phi\mapsto\phi^\Sigma$ is covariant: $\left(\psi\phi\right)^\Sigma=\psi^\Sigma\phi^\Sigma$.
\end{defi}

Let $\Theta_n$ denote the tree with exactly one node and $n$ inputs. Any tree $T$ with a root node of valence $n$ decomposes 
uniquely into $n$ trees $T_1,\ldots,T_n$ whose outputs are grafted onto the inputs of $\Theta_n$ (see picture below).
	$$
  \parbox{1cm}{$T=$}
  \parbox{4cm}
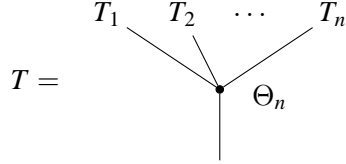
{
  \begin{tikzpicture}[
  every node/.style={fill=white,circle,inner sep=1pt},
  level distance=10mm,sibling distance=10mm]
  \node[white] {}[grow=north]
  child { node [draw,fill=black]  {}
  child { node {$T_n$} }
  child[missing] 
  child { node {$T_2$} }
  child { node {$T_1$} }
  };
  \draw(0.3,0.9)  node[right] {$\Theta_n $};
  \draw(0.1,2)  node[right] {$\ldots$};
  \end{tikzpicture}
  }
  $$
  We denote this grafting operation by
$$
T=\Theta_n\circ(T_1\oplus\ldots\oplus T_n).
$$

Any isomorphism $\phi:T\to T'$ has a similar decomposition
$$
\phi=\sigma\circ(\phi_1\oplus\ldots\oplus\phi_n)
$$
into isomorphisms $\sigma:\Theta_n\to\Theta_n$ and $\phi_i:T_\sigma(i)\to T'_{i}$. Since $\sigma$ only
permutes the inputs of $\Theta_n$ we usually denote $\sigma^\Sigma$ simply by $\sigma$.

Since the number of nodes and edges in each $T_i$ is strictly less than 
the number of nodes and edges in $T$, this decomposition is suited for inductive procedures.

For any collection $\mathcal{K}$ we define a functor $\underline{\mathcal{K}}:\mathbb{T}^{\op}\to\mathcal{S}$ inductively by 
mapping the trivial tree to the terminal object and putting
$$
\underline{\mathcal{K}}(T)=\underline{\mathcal{K}}(\Theta_n\circ(T_1\oplus\ldots\oplus T_n))=\mathcal{K}(n)\times 
\underline{\mathcal{K}}(T_1)\times\ldots\times\underline{\mathcal{K}}(T_n)
$$
On morphisms $\phi:T\to T'$ we define $\phi^\ast:\underline{\mathcal{K}}(T')\to\underline{\mathcal{K}}(T)$ inductively by
$$
\phi^\ast=(\sigma\circ(\phi_1\oplus\ldots\oplus\phi_n))^\ast=\sigma^\ast\times \phi^\ast_{\sigma(1)}\times\ldots\times\phi^\ast_{\sigma(n)}.
$$
which is determined by setting
$$
\sigma^\ast:\mathcal{K}(n)\to\mathcal{K}(n), \; a\mapsto a\cdot\sigma
$$
There is also a functor $\lambdaov:\mathbb{T}\to\mathcal{S}ets$ associating with each tree $T$ the set $\lambdaov(T)$ of 
bijections $\tau:\{1,2,\ldots,\In(T)\}\to \underline{\In(T)}$. On morphisms $\phi:T\to T'$ we define 
$$
\lambdaov(\phi):\lambdaov(T)\to\lambdaov(T), \; \tau\mapsto \In(\phi)\circ \tau.
$$
Since $\Sets$ is canonically included in $\Cat$, $\SSets$, and $\Top$ as the full 
subcategory of discrete objects, we can consider $\lambdaov$ as a functor $\lambdaov:\mathbb{T}\to\mathcal{S}$. 
The groupoid $\mathbb{T}$ is the disjoint sum of the groupoids $\mathbb{T}(n)=\{T\in\mathbb{T};\In(T)=n\}$ and
the free operad functor
$$
L:\Coll(\mathcal{S})\to\Opr(\mathcal{S})
$$
sends the collection $\scK$ to the operad whose underlying collection is the family of coends
$$
L\mathcal{K}(n)=\underline{\mathcal{K}}\otimes_{\mathbb{T}(n)}\lambdaov,\ \ n\in \mN.$$
Before we define the operad structure let us give an explicit description of $L\mathcal{K}(n)$.
An element of $L\mathcal{K}(n)$ is represented by a triple $(T,f,\tau)$ consisting of a tree $T$ with $n$ inputs, a 
function $f$ assigning to each node $\nu$ of $T$ an element $a\in\mathcal{K}(\In(\nu))$, and a bijection
$\tau:\underline{n}=\{1,2,\ldots,n\}\to\underline{\In(T)}$. We call $a$ the \textit{decoration} of $\nu$ 
and $i$ the \textit{label} of the input $\tau(i)$. We usually suppress $f$ and $\tau$ and speak of a decorated tree $T$ with input labels. 

\begin{leer}\label{free_1}\textit{Equivariance relation;}
We impose the following relation on the set of decorated trees $T$ with input labels. Let
	$$
  \parbox{1cm}{$T' = $}
  \parbox{4cm}
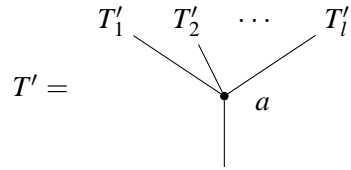
{
  \begin{tikzpicture}[
  every node/.style={fill=white,circle,inner sep=1pt},
  level distance=10mm,sibling distance=10mm]
  \node[white] {}[grow=north]
  child { node [draw,fill=black]  {}
  child { node {$T'_l$} }
  child[missing] 
  child { node {$T'_2$} }
  child { node {$T'_1$} }
  };
  \draw(0.3,0.9)  node[right] {$a$};
  \draw(0.1,2)  node[right] {$\ldots$};
  \end{tikzpicture}
  }
  $$
  be a subtree of $T$ above a node $\nu$ with decoration $a\in\mathcal{K}(l)$ and let $\sigma\in\Sigma_l$. 
Then $T$ is equivalent to the decorated tree ${}^\sigma T$ obtained from $T$ by replacing $T'$ by
	$$
  \parbox{1cm}{$T'' =$}
  \parbox{5cm}
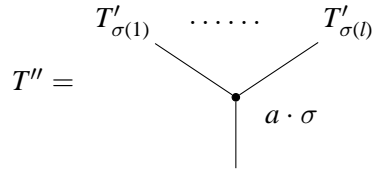
{
  \begin{tikzpicture}[
  every node/.style={fill=white,circle,inner sep=1pt},
  level distance=10mm,sibling distance=10mm]
  \node[white] {}[grow=north]
  child { node [draw,fill=black]  {}
  child { node {$T'_{\sigma(l)}$} }
  child[missing] 
  child[missing] 
  child { node {$T'_{\sigma(1)}$} }
  };
  \draw(0.3,0.75)  node[right] {$a\cdot\sigma$};
  \draw(-0.7,2)  node[right] {$\ldots\ldots$};
  \end{tikzpicture}
  }
  $$
\end{leer}

The elements of $L\mathcal{K}(n)$ are the equivalence classes of decorated trees with input labels with respect to this relation.

If $(T,f,\tau)$ represents an element $x$ in $L\mathcal{K}(n)$ and if $\sigma\in\Sigma_n$, we define $x\cdot \sigma$ to be represented
by $(T,f,\tau\circ\sigma)$. This defines the right action of $\Sigma_n$ on $L\mathcal{K}(n)$. Operad composition is defined 
by grafting decorated trees with input labels according to the 
labels: $T\circ(T_1\oplus\ldots\oplus T_n)$ is obtained by grafting $T_i$ on the input of $T$ labeled by $i$.

Let $\tau:\underline{n}\to\underline{\In(T)}$ be an input labeling of $T$, and suppose $\tau(i)$ is the $k$-th 
input of $T$ counted from left to right. Then we identify $\tau$ with the permutation $\tau\in\Sigma_n$ sending $i$ to $k$. Using this identification we obtain

\begin{prop}\label{free_2}
	 $L\mathcal{K}(n)=\underline{\mathcal{K}}\otimes_{\mathbb{T}(n)}\lambdaov=\coprod\limits_{[T]}
	\underline{\mathcal{K}}(T)\otimes_{\Aut (T)}\Sigma_n,\ \ n\ge 0$,

where the sum is indexed by isomorphism classes of trees in $\mathbb{T}(n)$.
\end{prop}

For later use we observe that Proposition \ref{free_2} is a special case of a more general result.

\begin{leer}\label{free_3}
 Let $\mathbb{G}$ be a groupoid and let $\underline{F}: \mathbb{G}^{\op}\to \scS$ and $\lambdaov:\mathbb{G}\to \scS$ be functors. Then
 $$\underline{F}\otimes_{\mathbb{G}} \lambdaov=\coprod\limits_{[G]}\underline{F}\otimes_{[G]} \lambdaov\cong \coprod\limits_{[G]}\underline{F}(G)\otimes_{\Aut(G)} \lambdaov(G)$$
 where the sum is indexed by isomorphism classes in $\mathbb{G}$. The coend $\underline{F}\otimes_{[G]} \lambdaov$ is taken over the elements in
 the class $[G]$. The isomorphism depends on the choice of representatives $G$ in the class $[G]$.
\end{leer}

\section{Rectifying $\hscO$-spaces}\label{simpleversion}

We start with our rectification construction for $\hscO$-spaces, which is easier
 to describe than the version we use for the $\Cat$ case.
Although this space version is simpler, it uses some of the same ingredients as our
subsequent rectification construction for $\hscO$-categories and will help to
motivate that construction.  In the process we give a simple variant of a rectification result of May 
and Thomason \cite[Theorem. 4.5]{May3}. We should also note that the construction we
define here is a variant of the $M$-construction of Boardman-Vogt
\cite[p. 134ff]{BV}.
\begin{leer}\label{rect1}
Let $\scO$ be an arbitrary operad in $\Top$. We are going to define a rectification functor
$$
M: \Top^{\hscO}\to \scO\mbox{-}\Top.
$$
Our construction starts with
a modification of the free operad construction. We inductively define a functor $\mathcal{L}\scO:\mT^{\op}
\to \Top$ by mapping the trivial tree to a point and putting
$$
\mathcal{L}\scO(T)=\mathcal{L}\scO(\Theta_n\circ (T_1\oplus\ldots\oplus T_n))=\scO(n)\times I^n\times 
\mathcal{L}\scO(T_1)\times\ldots\times \mathcal{L}\scO(T_n)
$$
where $I$ is the unit interval. On morphisms $\phi: T\to T'$ the functor is given by
$$
\phi^\ast=(\sigma\circ (\phi_1\oplus\ldots\oplus \phi_n))^\ast =
\sigma^\ast\times \phi_{\sigma(1)}^\ast\times\ldots\times \phi_{\sigma(n)}^\ast
$$
with
$$
\sigma^\ast:\scO(n)\times I^n\to \scO(n)\times I^n,\quad (a;t_1,\ldots,t_n)\mapsto 
(a\cdot \sigma;t_{\sigma(1)},\ldots,t_{\sigma(n)}).
$$
For $G: \hscO \to \Top$ there is a functor $\lambda = \lambda_G: \mT\to \Top$, sending the trivial tree to 
$G(\scO(1))$ and 
$\Theta_n\circ(T_1\oplus\ldots\oplus T_n)$ to $G(\In (T_1))\times\ldots \times G(\In (T_n))$. In particular,
$\lambda(\Theta_n)=G(\scO(1))^n$. On
morphisms $\sigma:\Theta_n\to \Theta_n$ it is defined by
$$
\lambda(\sigma):G(1)^n\to G(1)^n,\quad (g_1,\ldots,g_n)\mapsto (g_{\sigma^{-1}(1)},\ldots, g_{\sigma^{-1}(n)})
$$
and for $\phi=\sigma\circ (\phi_1\oplus\ldots\oplus \phi_n):T\to T'$ by 
$$
\begin{array}{rcl}
 \lambda(\phi):G(\In (T_1))\times\ldots \times G(\In (T_n))& \rightarrow &G(\In (T'_1))\times\ldots \times G(\In (T'_n))\\
 (g_i)_{i=1}^{n} & \mapsto & (G(\phi_{\sigma^{-1}(i)}^\Sigma)(g_{\sigma^{-1}(i)}))_{i=1}^{n}.
\end{array}
$$
Here recall that $\phi_{\sigma^{-1}(i)}: T_{\sigma^{-1}(i)}\to T"_i$ is in $\mT$ and $\phi_{\sigma^{-1}(i)}^\Sigma$ is the induced
inputs permutation (cf. Definition \ref{ind_perm}).
A natural transformation $G\to G'$ induces a natural transformation $\lambda_G\to \lambda_{G'}$.

Let $\widetilde{\mT}\subset \mT$ be the full subgroupoid of non-trivial trees. Restricting our functors to $\tmT$ the
 coend construction defines a functor 
$$
\scL\scO\otimes_{\tmT}\lambda_{(-)}:\Top^{\hscO}\to\Top,\quad G\mapsto \scL\scO\otimes_{\tmT}\lambda_G.
$$
The functor $M: \Top^{\hscO}\to \scO\mbox{-}\Top$ will be a quotient of this functor.
\end{leer}

\begin{leer}\label{rect2}
We find it helpful to view an element of $\scL\scO(T)$ as a triple  $(T,f,h)$ consisting of a 
tree $(T,f)$ with vertex decorations like in Section 3, and a length function $h$ assigning to each internal edge of $T$ a length 
in $I$.
We usually suppress $f$ and $h$ and speak of a decorated tree $T$ with lengths whose nodes are decorated by elements in $\scO$
and whose internal edges have a length label.
It will be clear from the context whether $T$ denotes a decorated tree with lengths or just a tree.
 Let $T$ have the form
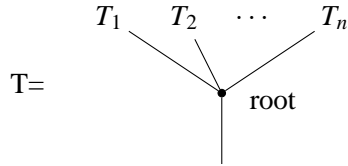
\begin{leer}\label{rect2a}
$$
\parbox{1cm}{T=}
\parbox{4cm}{
\begin{tikzpicture}[
every node/.style={fill=white,circle,inner sep=1pt},
level distance=10mm,sibling distance=10mm]
\node[white] {}[grow=north]
child { node [draw,fill=black]  {}
child { node {$T_n$} }
child[missing] 
child { node {$T_2$} }
child { node {$T_1$} }
};
\draw(0.3,0.9)  node[right] {root};
\draw(0.1,2)  node[right] {$\ldots$};
\end{tikzpicture}
}
$$
Here $T_i$ is allowed to be the trivial tree.
\end{leer}
We define
$$
\begin{array}{rcl}
V(G,T) &=& \mathcal{L}\scO(T)\times G(\In(T))\\
U(G,T) &=& \mathcal{L}\scO(T)\times G(\In(T_1))\times\ldots\times G(\In(T_n))
\end{array}
$$
$$\scL\scO\otimes_{\tmT}\lambda_G =\left(\coprod_T U(G,T)\right)/\sim$$
\end{leer}
where the unions is taken over all trees in $\tmT$ and the relations are as follows:
\begin{leer}\label{rect3} \textit{Equivariance relations:} 
It is helpful to consider $U(G,T)$ as
$$
U(G,T)=\scO(n)\times I^n\times V(G,T_1)\times\ldots\times V(G,T_n)
$$
where $(t_1,\ldots,t_n)\in I^n$ are the lengths of the incoming edges of the root from left to right and $\scO(n)$ is the space of root decorations.
\begin{enumerate}
	\item \textit{Root equivariance:} Let $\sigma\in\Sigma_n$, then
	\begin{multline*}
	(a;t_1,\ldots,t_n;(T_1,g_1),\ldots,(T_n,g_n))\\
	\qquad \sim (a\cdot\sigma;t_{\sigma(1)},\ldots,t_{\sigma(n)};
	(T_{\sigma(1)},g_{\sigma(1)}),\ldots,(T_{\sigma(n)},g_{\sigma(n)}))
	\end{multline*}
	\item $T_i$\textit{-equivariance:} $T_i$-equivariance is a relation on the factor $V(G,T_i)$. We use the notation
	of \ref{free_1} with the difference that the internal edges of our trees have a length label. As in \ref{free_1} let $T'$ be the
	subtree above a node $v$ of valence $l$ of $T_i$ decorated by $a$. Let
   $\sigma\in\Sigma_l$ and let ${}^\sigma T_i$ be obtained from $T_i$ as in \ref{free_1}. Then $\sigma$ determines an isomorphism
   $\phi: T_i\to {}{}^\sigma T_i$ of underlying trees in $\tmT$, and $T_i$-equivariance is the relation
$$
   (T_i;g_i)\sim({}^\sigma T_i;G(\phi^\Sigma)(g_i)).
   $$
 \end{enumerate}  
 \end{leer}  
\begin{defi}\label{rect4} The functor
$$M: \Top^{\hscO}\to \scO\mbox{-}\Top$$
is obtained from the functor $\scL\scO\otimes_{\tmT}\lambda_{(-)}$ by imposing the 
 following relations. Let $T$ be a decorated tree with lengths of the form \ref{rect2a}.
   
\begin{enumerate}
\item \textit{Shrinking an internal edge:} An internal edge $e$ of length 0 may be shrunk
\begin{center}
\begin{tikzpicture}[>=triangle 60]
\draw[fill](0,0) circle(1pt) node[right=2mm] {node $v$ with decoration $a$};
\draw[fill](0,2) circle(1pt) node[right=2mm] {node $w$ with decoration $b$};
\draw[] (0,0) -- (0,2);
\draw[->] (0,2) -- (0,1);
\draw(0,1.1)  node[left] {$e\ \ =\ \ $} node[right=2mm] {edge e with length $0$};
\draw[] (0,0) -- (0,2);
;
\end{tikzpicture}
\end{center}
Let $T'$ be obtained from $T$ by shrinking $e$. If $e$ is the $i$-th input of $v$ counted from left to right the new node in $T'$ 
is decorated by $a\circ(\id_{i-1}\oplus b\oplus\id_{\In(v)-i})$.
\begin{enumerate}
	\item $v$ \textit{is not the root:} Then
	$$
	(T;g_1,\ldots,g_n)\sim(T';g_1,\ldots,g_n), \qquad g_i\in G(\In(T_i))
	$$
	\item $v$ \textit{is the root:} Then $T_i$ has the form
	
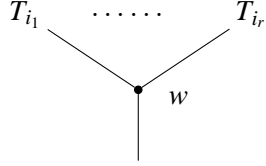
\begin{center}
  \begin{tikzpicture}[
  every node/.style={fill=white,circle,inner sep=1pt},
  level distance=10mm,sibling distance=10mm]
  \node[white] {}[grow=north]
  child { node [draw,fill=black]  {}
  child { node {$T_{i_r}$} }
  child[missing] 
  child[missing] 
  child { node {$T_{i_1}$} }
  };
  \draw(0.3,0.9)  node[right] {$w$};
  \draw(-0.7,2)  node[right] {$\ldots\ldots$};
  \end{tikzpicture}
  \end{center}
  and $e$ is the outgoing edge of $w$.  If $w$ is not a stump,
  shrinking $e$ makes the incoming edges of $w$ into incoming edges of the root of $T'$.
  Let $\tau_j:\underline{\In(T_{i_j})}\subset\underline{\In(T_i)}$ be the inclusion, then there is a map
  $$
  \tau^\ast:G(\In(T_i))\to\prod^r_{j=1}G(\In(T_{i_j}))
  $$
  whose $j$-the component is $G(\tau^\ast_j)$. We have the relation
  $$
  (T;g_1,\ldots, g_n)\sim (T';g_1,\ldots,g_{i-1}, \tau^\ast(g_i),g_{i+1},\ldots,g_n)
  $$
  If $w$ is a stump, $In(T_i)=\emptyset$ and we impose the relation
  $$(T;g_1,\ldots, g_n)\sim (T';g_1,\ldots, g_n).$$
\end{enumerate}
\item \textit{Chopping an internal edge:} An internal edge $e$ of length 1 may be chopped off.  
Let $e$ be as above, but of length $1$. Let $T''$ be the subtree of $T$ with root $w$.
 Then $T''$ is a subtree of some $T_i$. Let $T'$ be obtained from $T$ by deleting the subtree $T''$.
 Composing all node decorations of $T''$ using the operad composition gives us an element 
$c\in\scO(\In(T''))$. We label the inputs of $T_i$ from left to right by $1$ to $\In(T_i)$. 
Then the inputs of $T''$ form a subinterval $s+1,s+2,\ldots, s+t$ with $t=\In(T'')$. Define
$$
\widehat{c}=\id_s\oplus c\oplus\id_{\In(T_i)-s-t}\in\scO\subset \hscO.
$$
We have the relation
$$
(T;g_1,\ldots,g_k)\sim (T';g_1,\ldots,g_{i-1}, G(\widehat{c})(g_i), g_{i+1},\ldots,g_k)
$$
In particular, if $w$ is a stump then $c=b\in \scO(0)$, $\In(T'')=\emptyset$ so
that $t=0$, and $\In(T')=\In(T)+1$. 
\end{enumerate}
\end{defi}

\begin{prop}\label{rect5} $M(G)$ has an $\scO$-algebra structure, and we obtain a functor
$$M: \Top^{\hscO}\to \scO\mbox{-}\Top.$$
\end{prop}

\begin{proof}
Let $x_i\in M(G)$, $i=1,\ldots,n$, be represented by $(T_i; g_{i1},\ldots,
g_{ik_i})$ and let $a\in \scO(n)$. Then $a(x_1,\ldots,x_n)$ is represented by
$$
(T; g_{11},\ldots,g_{1k_1},\ldots,g_{n1},\ldots,g_{nk_n})
$$
where $T$ is obtained from $T_1,\ldots,T_n$ by grafting the roots of the $T_i$ together to a single root.
 If the root of $T_i$ is decorated by $b_i$, the new root is decorated by
$$
a\circ(b_1\oplus\ldots\oplus b_n).
$$
\end{proof}

We want to compare the $\hscO$-space $\widehat{M(G)}$ associated with the $\scO$-algebra $M(G)$ with the original $\hscO$-space $G$. 
For this purpose we define an $\hscO$-space 
$$
Q(G):\hscO\to\Top, \quad n\mapsto Q_n(G)
$$
by
$$
Q_n(G)=\left(\coprod \mathcal{L}\scO(T_1)\times\ldots\times\mathcal{L}\scO(T_n)\times G(\In(T_1)+\ldots+\In(T_n))\right)/\sim
$$
where the union is taken over all $n$-tuples $(T_1,\ldots,T_n)$ of trees in $\tmT$. The relations are
\begin{enumerate}
	\item \textit{Shrinking an internal edge:} An internal edge $e$ of length $0$ in any of the trees may be shrunk as
	explained in Definition \ref{rect4}.1a, which makes sense even if $e$ is a root edge. 
	\item \textit{Chopping an internal edge:} Any internal edge $e$ of length $1$ in any of the trees may be chopped
	as explained in Definition \ref{rect4}.2 with the difference that $\widehat{c}$ is formed using all inputs rather than only the ones
        of $T_i$.
	\item\textit{ Equivariance:} $T_i$-equivariance as explained in \ref{rect3}.2 holds for each tree $T_1,\ldots,T_n$ 
         and the relation reads
$$(T_1,\ldots,T_n;g)\sim (T_1,\ldots,{}^\sigma T_i,\ldots,T_n;G(\id\times\ldots\times \phi^\Sigma\times\ldots\times\id)(g))
$$
\end{enumerate}

This defines $Q(G)$ on objects.\\
 For $\sigma\in\Inj(k,l)$ the map $Q(G)(\sigma^\ast):Q_l(G)\to Q_k(G)$ 
is given by the projections
$$
\mathcal{L}\scO(T_1)\times\ldots\times\mathcal{L}\scO(T_l)\to\mathcal{L
}\scO(T_{\sigma(1)})\times\ldots\times\mathcal{L}\scO(T_{\sigma(k)})
$$
and the map $G(\sigma(\In(T_1),\ldots,\In(T_l))^\ast)$ defined in \ref{operads4}. If $k_1+\ldots+k_r=m$ and 
$$
\alpha=(\alpha_1,\ldots,\alpha_r)\in\scO(k_1,1)\times\ldots\times\scO(k_r,1)\subset\hscO(m,r)
$$
then $Q(G)(\alpha):Q_m(G)\to Q_r(G)$ maps a representing tuple $(T_1,\ldots,T_m;g)$, 
where each $T_i$ is a decorated tree with lengths and $g\in G(\In(T_1)+\ldots+\In(T_m))$,
 to the element represented by $(T'_1,\ldots,T'_r;g)$. If $p=k_1+\ldots+k_{i-1}$ then $T'_i$ is
 obtained from $T_{p+1},\ldots,T_{p+k_i}$ by grafting their roots together and decorating the 
root of $T'_i$ by $\alpha_i\circ(\beta_1\oplus\ldots\oplus\beta_{k_i})$ where $\beta_j$ is the 
root decoration of $T_{p+j}$.

Like $M(G)$, the space $Q_n(G)$ is the quotient of a coend, namely the coend of the functor 
$$\scL\scO_{Q_n}: (\tmT^{\op})^n \xrightarrow{\scL\scO^n} \Top^n \xrightarrow{\product} \Top$$
and the functor
$$\lambda_{Q_n}: \tmT^n \to \Top,\quad (T_1,\ldots,T_n)\mapsto G(\In(T_1)+\ldots+\In(T_n)).$$

\begin{theo}\label{rect6} There are maps of $\hscO$-spaces, natural in $G$,
$$
\widehat{M(G)} \stackrel{\tau}{\longleftarrow} Q(G) \stackrel{\varepsilon}{\longrightarrow} G
$$
such that
\begin{enumerate}
	\item each $\varepsilon_n:Q_n(G)\to G(n)$ is a homotopy equivalence,
	\item if $\scO$ is $\Sigma$-free and $G$ is special, each $\tau_n:Q_n(G)\to M(G)^n$ is a homotopy
	equivalence.
\end{enumerate}
\end{theo}

\begin{proof}
The map $\varepsilon_n:Q_n(G)\to G(n)$ is defined by chopping the roots of each tree.
 This makes sense in this case although roots are not internal edges. By construction, the 
$\varepsilon_n$ define a map $\varepsilon:Q(G)\to G$ of $\hscO$-spaces. 
Each $\varepsilon_n :Q_n(G)\to G(n)$ has a section
$$
s_n:G(n)\to Q_n(G),\quad g\mapsto (\Theta_1,\ldots,\Theta_1;g)
$$
with $\id\in\scO(1)$ as node decoration of $\Theta_1$. Let $T(t)$ be the tree obtained from $T$ by putting 
$T$ on top of $\Theta_1$ and giving the newly created internal edge the length $t$. Then for
 $(T_1,\ldots,T_n;g)\in\mathcal{L}\scO(T_1)\times\ldots\times\mathcal{L}\scO(T_n)\times G(\In(T_1)+\ldots+\In(T_n))$ 
we have
$$
(T_1(0),\ldots,T_n(0);g)\sim (T_1,\ldots,T_n;g)
$$
by the shrinking relation, and
$$
(T_1(1),\ldots,T_n(1);g)\sim s_n\varepsilon_n(T_1,\ldots,T_n;g)
$$
by the chopping relation. Hence $t\mapsto(T_1(t),\ldots,T_n(t);g)$ defines a homotopy 
from $\id_{Q_n(G)}$ to $s_n\circ\varepsilon_n$.

We define
$$
\tau:Q_n(G)\longrightarrow M(G)^n,\quad(T_1,\ldots,T_n;g)\longmapsto\left(T_i;G(\sigma^\ast_{i,1})(g),\ldots,G(\sigma^\ast_{i,k_i})(g) \right)^n_{i=1}
$$
if $T_i$ is of the form $T_i=\Theta_{k_i}\circ (T_{i,1}\oplus\ldots\oplus T_{i,k_i})$, and where 
 $\sigma_{i,j}: \underline{\In(T_{i,j})}\subset   \underline{\In(T_i)}\subset\underline{\In(T_1)+\ldots+\In(T_n)}$ is the canonical
 inclusion. By construction, the $\tau_n$ define a map of $\hscO$-spaces.

We now prove the second statement of the theorem. So assume that $G$ is special and $\scO$ is $\Sigma$-free. Then the map
$$
\pi_n: (G(\sigma^\ast_{1,1}),\ldots,G(\sigma^\ast_{n,k_n})): G(\In(T_1)+\ldots+\In(T_n))\to
\prod^n_{i=1}\prod^{k_i}_{j+1}G(\In(T_{i,j}))
$$
is a homotopy equivalence. For notational convenience we denote $G(\In(T_1)+\ldots+\In(T_n))$ by $G_Q(T_1,\ldots,T_n)$,
$\prod^n_{i=1}\prod^{k_i}_{j+1}G(\In(T_{i,j}))$ by $G_M(T_1,\ldots,T_n)$ and $\prod^n_{i=1}\Aut(T_i)$ by
$\Aut(T_1,\ldots,T_n)$. Similarly we denote 
$$\scL\scO_{Q_n}(T_1,\ldots,T_n)=\scL\scO(T_1)\times\ldots\times \scL\scO(T_n)$$
by $\scL\scO(T_1,\ldots,T_n)$.

 By Proposition \ref{free_3}, $M(G)^n$ is a quotient of 
$$\coprod_{([T_1],\ldots,[T_n])}(\scL\scO\times_{\Aut}G_M)(T_1,\ldots,T_n)$$
where
$$(\scL\scO\times_{\Aut}G_M)(T_1,\ldots,T_n)=\scL\scO(T_1,\ldots,T_n)\times_{\Aut(T_1,\ldots,T_n)}
G_M(T_1,\ldots,T_n)$$
and $Q_n(G)$ is a quotient of 
$$
\coprod_{([T_1],\ldots,[T_n])}(\scL\scO\times_{\Aut}G_Q)(T_1,\ldots,T_n)$$
where
$$(\scL\scO\times_{\Aut}G_Q)(T_1,\ldots,T_n)=\scL\scO(T_1,\ldots,T_n)
\times_{Aut(T_1,\ldots,T_n)} G_Q(T_1,\ldots,T_n).$$
In both cases the sum is indexed by the isomorphism classes in $\widetilde{\mT}^n$.

Hence the proof reduces to showing that
$$id\times_{\Aut} \pi_n:\coprod_{([T_1],\ldots,[T_n])}(\scL\scO\times_{\Aut}G_Q)(T_1,\ldots,T_n)\to
\coprod_{([T_1],\ldots,[T_n])}(\scL\scO\times_{\Aut}G_M)(T_1,\ldots,T_n)
$$
induces a homotopy equivalence $Q_n(G)\to M(G)^n$. (This part of the proof relies on certain technical facts about
numerable principal bundles that we list in \ref{rect_6a}.)

We choose a representative $T$ in each isomorphism class $[T]$ and filter both spaces. Let $F_r(Q)$ and $F_r(M)$ be the subspaces of
$Q_n(G)$ and $ M(G)^n$ of those points which can be represented by elements for which the $(T_1,\ldots,T_n)$-part consists of trees whose total number 
of internal edges is less than or equal to $r$. We prove by induction that the above map induces a homotopy equivalence
$F_r(Q)\to F_r(M)$ for all $r$, which in turn implies the result.

$F_0(Q)$ is the disjoint union of spaces 
$$(\scO(k_1)\times \ldots\times \scO(k_n))\times_{\Sigma_{k_1}\times\ldots\times \Sigma_{k_n}} G(k_1+\ldots + k_n)$$
and $F_0(M)$ is the disjoint union of spaces 
$$(\scO(k_1)\times \ldots\times \scO(k_n))\times_{\Sigma_{k_1}\times\ldots\times \Sigma_{k_n}} G(1)^{k_1+\ldots + k_n}.$$
Here observe that $\Aut(\Theta_k)=\Sigma_k$. Since $\scO(k_1)\times \ldots\times \scO(k_n)$ is a numerable principal
$(\Sigma_{k_1}\times\ldots\times \Sigma_{k_n})$-space, $\id\times \pi_n$ defines a homotopy equivalence $F_0(Q)\to F_0(M)$ by
\ref{rect_6a}(2).

Now assume that we have shown that $\id\times_{\Aut} \pi_n$ induces a homotopy equivalence $F_{r-1}(Q)\to F_{r-1}(M)$. We obtain
$F_r(Q)$ from $F_{r-1}(Q)$ and $F_r(M)$ from $F_{r-1}(M)$ by attaching spaces $(\scL\scO\times_{\Aut}G_Q)(T_1,\ldots,T_n)$
respectively $(\scL\scO\times_{\Aut}G_M)(T_1,\ldots,T_n)$, where $(T_1,\ldots,T_n)$ have exactly a total of $r$ internal edges.
 In both cases, an element in the attached spaces
represents an element of lower filtration if and only if an internal edge in $(T_1,\dots,T_n)$ is of length $0$ or $1$, because in these cases
the shrinking respectively the chopping relation applies. Let $D(T_1,\ldots,T_n))\subset \scL\scO(T_1,\ldots,T_n))$ be the 
subspace of such decorated trees. The inclusion of this subspace is an $\Aut(T_1,\ldots,T_n)$-equivariant cofibration (cf. the product
pushout theorem \cite[p. 233]{BV}). Consider the diagram
$$\xymatrix{
F_{r-1}(Q)\ar[d] & (D\times_{\Aut}G_Q)(T_1,\ldots,T_n)\ar[l]_(.63){\proj}\ar[r]^{i_Q}
\ar[d] & (\scL\scO \times_{\Aut}G_Q)(T_1,\ldots,T_n) \ar[d] \\
F_{r-1}(M) & (D\times_{\Aut}G_M)(T_1,\ldots,T_n)\ar[l]_(.63){\proj}\ar[r]^{i_M}
 & (\scL\scO \times_{\Aut}G_M)(T_1,\ldots,T_n)
}
$$
where the vertical maps are induced by $\id\times_{\Aut} \pi_n$. The maps $i_Q$ and $i_M$ are closed cofibrations (cf. \cite[p. 232]{BV}). We will prove in Lemma \ref{rect_6b}
that $\underline{\scO}(T)$ is a numerable principal $\Aut(T)$-space. (Recall that $\underline{\scO}:\mathbb{T}^{\op}\to\Top$ was
defined in Section \ref{free} for the collection $\scK=\scO$.)
Hence $\underline{\scO}(T_1)\times\ldots\times \underline{\scO}(T_n)$ is a numerable
principal $\Aut(T_1,\ldots,T_n) $-space (cf. \ref{rect_6a}(4)). Since there are equivariant maps 
$$D(T_1,\ldots,T_n))\rightarrow \scL\scO(T_1,\ldots,T_n))\xrightarrow{\textrm{forget}}\underline{\scO}(T_1)\times\ldots\times \underline{\scO}(T_n)$$
the spaces $D(T_1,\ldots,T_n)$ and $\scL\scO(T_1,\ldots,T_n)$ are numerable principal $\Aut(T_1,\ldots,T_n) $-spaces (cf. \ref{rect_6a}(1)) Hence the 
vertical maps of the diagram are homotopy equivalences (cf. \ref{rect_6a}(2)), and the gluing lemma implies that
$F_r(Q)\to F_r(M)$ is a homotopy equivalence.
\end{proof}

\begin{leer}\label{rect_6a}
 \textbf{Facts about numerable principal $\Gamma$-spaces}
 The following results are either fairly obvious or can be found in the appendix of \cite{BV}. Let $\Gamma$ be a discrete group and $X$ a numerable
 principal right $\Gamma$-space.
 \begin{enumerate}
  \item If $f:Y\to X$ is a $\Gamma$-equivariant map, then $Y$ is a numerable principal $\Gamma$-space. Moreover, $f$ is an equivariant homotopy equivalence
  if and only if it is an ordinary homotopy equivalence of underlying spaces.
  \item If $f:Y\to Z$ is an equivariant map of left $\Gamma$-spaces, which is an ordinary homotopy equivalence of underlying spaces, then
  $\id\times_\Gamma f:X\times_\Gamma Y\to X\times_\Gamma Z$ is a homotopy equivalence.
  \item If $H$ is a subgroup of $\Gamma$, then $X$ is a numerable principal $H$-space.
  \item If $Y$ is a numerable principal right $\Gamma'$-space, then $X\times Y$ is a numerable principal right $\Gamma\times\Gamma'$-space.
  \end{enumerate}
\end{leer}

\begin{lem}\label{rect_6b}
If $\scO$ is a $\Sigma$-free topological operad, then $\underline{\scO}(T)$ is a numerable principal $\Aut(T)$-space.
\end{lem}
\begin{proof}
We prove this by an inductive argument. If $T=\Theta_n$, then $\Aut(T)=\Sigma_n$ and $\underline{\scO}(T)=\scO(n)$. By assumption, 
$\scO(n)$ is a numerable principal $\Sigma_n$-space.

Now let $T=\Theta_n\circ (T_1\oplus\ldots\oplus T_n)$. The following calculation of $\Aut(T)$ is taken from \cite[p. 815]{Berger2}.
By choosing $T$ appropriately in $[T]$ we may assume that it has the form
$$T=\Theta_n\circ(T^1_1\oplus\ldots\oplus T^1_{k_1}\oplus\ldots\oplus T^l_1\oplus\ldots\oplus T^l_{k_l})$$
where $T^i_1,\ldots,T^I_{k_i}$ are copies of a planar tree $T^i$ and $T^i$ and $T^j$ are not isomorphic in $\mT$ for $i\neq j$.
Then $\Aut(T)$ is the semi-direct product
$$ \Aut(T)\cong \Aut(T^1)^{k_1}\times \ldots\times \Aut(T^l)^{k_l})\rtimes (\Sigma_{k_1}\times \ldots\times \Sigma_{k_l})=\Gamma_T
\rtimes \Sigma_T$$
where $\Sigma_{k_i}$ acts on $\Aut(T^i)^{k_i}$ by permuting the factors.

$\scO(n)$ is a numerable principal $\Sigma_T$-space because $\Sigma_T$ is a subgroup of $\Sigma_n$. By induction 
$\underline{\scO}(T^1)^{k_1}\times \ldots\times \underline{\scO}(T^l)^{k_l}$ is a numerable principal $\Gamma_T$-space. Denote  $\Theta_n$ by $T^0$. 
By \cite[Appendix 3.2]{BV} there are open covers $\scU^i=\{U^i_\alpha; \ \alpha\in A^i\}$ of $\underline{\scO}(T^i)$, $i=0,\ldots,l$ with
 subordinate partitions of unity $\{f^i_\alpha:\underline{\scO}(T^i)\to [0,1],\ \alpha\in A^i\}$
such that $U^i_\alpha\cdot h\cap U^i_\alpha=\emptyset$ for all $h\in H_i$ different from the unit, where $H_0=\Sigma_T$ and
$H_i=\Aut(T^i)$ for $i=1,\ldots,l$. The open cover
$$\scV=\{\{U^0_1\times U^1_1\ldots\times U^1_{k_1}\times\ldots\times U^l_1\times\ldots\times U^l_{k_l}\}$$
of $\underline{\scO}(T)$, where $U^i_j$ runs through the elements of $\scU^i$, satisfies the condition that
$$(U^0_1\times U^1_1\ldots\times U^l_{k_l})\cdot h\cap (U^0_1\times U^1_1\times\ldots\times U^l_{k_l})=\emptyset$$
for all $h\in\Gamma_T\rtimes \Sigma_T$ different from the unit. The product numeration obtained from the $f^i_\alpha$ provides a 
partition of unity subordinate to $\scV$. Now the the lemma follows from \cite[Appendix 3.2]{BV}.
\end{proof}

If $X$ is an $\scO$-algebra and $G=\widehat{X}$ then, by inspection, $Q_n(G)\cong Q_1(G)^n$, and  the $\hscO$-structure on
$Q(G)$ defines an $\scO$-algebra structure on $Q_1(G)$. The map $\tau_1: Q_1(G)\to M(G)$ is a homeomorphism of $\scO$-algebras, 
and $\varepsilon_1: Q_1(G)\to G(1)=X$ is a weak equivalence of $\scO$-algebras. Composing $\varepsilon_1$ with the inverse of
$\tau_1$ we obtain:

\begin{prop}\label{rect7}
 If $X$ is an $\scO$-algebra in $\Top$ then there is a natural weak equivalence of $\scO$-algebras
$$\overline{\varepsilon}: M(\widehat{X})\to X.$$
\end{prop}
\section{Tree-indexed diagrams}\label{diagram}

In order to adapt the rectification construction described in Section \ref{simpleversion}
to the case of categories, we will recast the topological version described there
into a homotopy colimit construction of a certain diagram.  The same diagram
makes sense in $\Cat$, where we will apply the Grothendieck construction, which is
the analog of the homotopy colimit in $\Cat$.

\begin{leer}\label{diagram1}
\textbf{The indexing category} $\scI$: As in the previous section, our construction
involves trees with a root vertex. So 
the objects of $\scI$ are the isomorphism classes $[T]$ of planar trees in $\tmT$. 
The shrinking and chopping relations of Definition \ref{rect4} correspond to
morphisms in the diagram to be constructed. So the generating morphisms of $\scI$ are of two types.

(1) Shrinking an internal edge.\\
(2) Chopping off a subtree above any node $\nu$ of a tree:\\
\raisebox{-20pt}{\qquad\qquad}
$
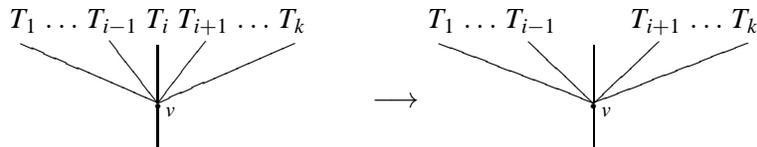
\xymatrix@=5pt@M=-1pt@W=-1pt{
T_1 &\dots &T_{i-1} &T_i &T_{i+1}&\dots &T_k\\
\ar@{-}[ddddrrr]& &\ar@{-}[ddddr] &\ar@{-}[dddd] &\ar@{-}[ddddl] &&\ar@{-}[ddddlll]\\
\\
\\
\\
&&&{\scriptscriptstyle\bullet}\ar@{-}[ddd]^(.2){v}\\
\\
\\
&&&
}$
\raisebox{-30pt}{$\qquad\longrightarrow$}
$\xymatrix@=5pt@M=-1pt@W=-1pt{
T_1 &\dots &T_{i-1} &\qquad &T_{i+1}&\dots &T_k\\
\ar@{-}[ddddrrr]& &\ar@{-}[ddddr] &\ar@{-}[dddd] &\ar@{-}[ddddl] &&\ar@{-}[ddddlll]\\
\\
\\
\\
&&&{\scriptscriptstyle\bullet}\ar@{-}[ddd]^(.2){v}\\
\\
\\
&&&
}$\newline
That is, the subtree $T_i$ in the original tree is replaced by a single input edge in the new tree.
\end{leer}
To define a general morphism in $\scI$ we introduce the notion of a \textit{marked tree}. A marked tree is 
a planar tree $S$ with a marking of 
some (possibly none) of its internal edges with either the symbol $s$ or the symbol
$c$ subject to the constraint that an edge which is anywhere above an edge marked $c$ is left
unmarked. A morphism in $\scI$ is an isomorphism class of marked trees with respect to non-planar
isomorphisms respecting the marking. The source of such a morphism $f$ is the isomorphism class of 
the underlying unmarked planar tree. Let $S$ be a marked tree representing $f$ and let $T$ be the unmarked tree
obtained from $S$ by first chopping off the branches above every edge marked $c$.
[Note that this map would discard any markings of edges above such an
edge, which accounts for the constraint.] Then one shrinks all edges marked $s$.
The isomorphism class of $T$ is the target of $f$. By construction, $[T]$ is independent of the choice of
the representative $S$ of $f$.
In most cases, there is at most one morphism between objects of $\scI$.  However there are
exceptions.  For instance
$$
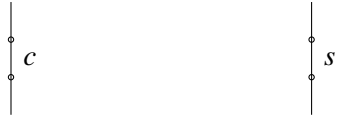
\begin{tikzpicture}
\draw (0,0) -- (0,0.5) -- (0,1) -- (0,1.5);
\draw(0,0.5) circle (1pt);
\draw(0,1) circle (1pt);
\draw[] (0,0.75)  node[right] {$c$};
\end{tikzpicture}
\hspace{20ex}
\begin{tikzpicture}
\draw (0,0) -- (0,0.5) -- (0,1) -- (0,1.5);
\draw(0,0.5) circle (1pt);
\draw(0,1) circle (1pt);
\draw[] (0,0.75)  node[right] {$s$};
\end{tikzpicture}
$$
represent distinct morphisms with the same source and target.

The following remark will make the definition of the composition easy.
\begin{rema}\label{diagram2}
In the sequel we will need to prescribe a consistent way of representing simplices
in the nerve of $\scI$ by a chain of marked planar trees.
Given an $n$-simplex
$$[T_0]\to [T_1]\to \ldots \to [T_n]$$
in the nerve of $\scI$. Pick a planar representative $T_0\in[T_0]$. Any morphism $[T_0]\to[T_1]$ is represented by a
marking of $T_0$. By applying the edge shrinking and chopping specified by the marking of $T_0$, we obtain
a well-defined planar representative $T_1$ of $[T_1]$. 
Now apply the same procedure to the map $[T_1]\to [T_2]$ and carry on to obtain a sequence
$$T_0\to T_1\to \ldots \to T_n \quad\textrm {representing}\quad [T_0]\to [T_1]\to \ldots \to [T_n],$$
where the maps  $T_i\to T_{i+1}$ are given by a marking of $T_i$.
 If we had picked a different representative $T_0'\in[T_0]$,
then there is an isomorphism $\phi:T_0 \to T_0'$ in $\tmT$, which transports the marking of $T_0$ to a marking of $T_0'$,
and the marked tree$T_0'$ also represents the morphism $[T_0]\to [T_1]$ .  Clearly
$\phi$ can be extended to the whole sequence of representatives in a
unique way.\end{rema}

To define the composition of $f:[S]\to [T]$ and $g:[T]\to [U]$, we take a representing chain
$S\to T\to U$ with a marking of $S$ and a marking of $T$. Let $E(S)$ and $E(T)$ be the sets of internal
edges of $S$ respectively $T$. Since $T$ is obtained from $S$ by shrinking and chopping off internal edges
we may consider $E(T)$ as a subset of $E(S)$. Observe that the marked edges of $S$ do not lie in $E(T)$. So
the marking of $T$ defines a marking on edges of $S$ which have not been marked before. We now erase in 
this larger marking any mark above an edge marked with $c$ to satisfy our constraint. The resulting marking
of $S$ represents the composition $g\circ f$.

\begin{leer}\label{diagram3}
 \textbf{The diagram:} Let $\scO$ be an arbitrary operad in $\scS$ and let $G:\hscO\to \scS$ be an $\hscO$-diagram in $\scS$,
 where $\scS$ is $\Cat,\ \Top, \ \Sets$ or $\SSets$. We are going to define a diagram
 $$F^G:\scI \to \scS.$$
 We recall the functor $\underline{\scO}: \mT^{\op}\to \scS$
 from Section \ref{free}. The definition of the functor $\lambda_G:\tmT\to \Top$ in \ref{rect1} makes also sense if we replace
 $\Top$ by $\scS$. We define
 $$F^G([T])=\underline{\scO}\otimes_{[T]}\lambda_G,$$
 the coend obtained by restricting of the functors to the isomophism class $[T]\subset \tmT$.
 
 As in \ref{rect2} we have the following explicit description of an element in $F^G([T])$. The object
 $$W(G,T)= \underline{\scO}(T)\times G(\In(T))$$
 replaces $V(G,T)$. If $T$ has the form \ref{rect2a}
we define
$$F^G([T])=\left(\coprod_{T\in[T]}\scO(n)\times \prod_{i=1}^n W(G,T_i)\right)/\sim \ \ = \left(\coprod_{T\in[T]}\underline{\scO}(T)\times \prod_{i=1}^nG(\In(T_i))\right)/\sim$$
 where the relation is the equivariance relation \ref{rect3} with the factor $I^n$ dropped.
 \end{leer}
Next we describe $F^G$ on the generating morphisms of $\scI$.

(1) Suppose $\alpha: [T]\to[T']$ is shrinking a bottom edge of $[T]$. So $\alpha$ is represented by $T$ with a single marking $s$
of the edge connecting the root node to a subtree $T_i$ of $T$.

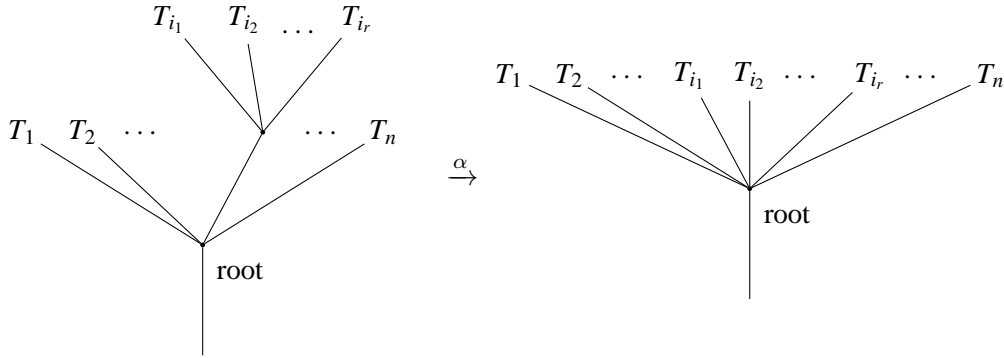
\begin{minipage}{5.8cm}
\begin{tikzpicture}
[level distance=15mm,
every node/.style={fill=white,circle,inner sep=0.5pt},
level 1/.style={sibling distance=10mm,nodes={fill=black}},
level 2/.style={sibling distance=8mm,nodes={fill=white}},
level 3/.style={sibling distance=5mm,nodes={fill=white}},
level 4/.style={sibling distance=10mm,nodes={fill=white}}]
\node[white] {}[grow=north]
child {node[draw] (above node) {}
child {node{$T_n$}} 
child [missing] {}
child {node[draw, fill=black] {}
child {node {$T_{i_r}$}}
child [missing] {}
child [missing] {}
child {node {$T_{i_2}$}}
child [missing] {}
child {node {$T_{i_1}$}}
}
child [missing] {}
child [missing] {}
child {node {$T_2$}}
child {node {$T_1$}}
};
\draw[] (1.6,2.65) node [above] {$\cdots$};
\draw[] (-0.8,2.65) node [above] {$\cdots$};
\draw[] (1.3,4) node [above] {$\cdots$};
\draw[] (0.5,0.8) node [above] {root};
\end{tikzpicture}
\end{minipage}
$\xrightarrow{\alpha}$
\begin{minipage}{6.0cm}
\begin{tikzpicture}
[level distance=15mm,
every node/.style={fill=white,circle,inner sep=0.5pt},
level 1/.style={sibling distance=10mm,nodes={fill=black}},
level 2/.style={sibling distance=8mm,nodes={fill=white}},
level 3/.style={sibling distance=5mm,nodes={fill=white}}]
\node[white] {}[grow=north]
child {node[draw] (above node) {}
child {node {$T_n$}}
child [missing] {}
child {node {$T_{i_r}$}}
child [missing] {}
child {node {$T_{i_2}$}}
child {node {$T_{i_1}$}}
child [missing] {}
child {node {$T_2$}}
child {node {$T_1$}}};
\draw[] (2.3,2.65) node [above] {$\cdots$};
\draw[] (0.7,2.65) node [above] {$\cdots$};
\draw[] (-1.6,2.65) node [above] {$\cdots$};
\draw[] (0.5,0.8) node [above] {root};
\end{tikzpicture}
\end{minipage}

The corresponding morphism $F^G(\alpha)$ is induced by the map
$$ 
\underline{\scO}(T)\times \prod_{j=1}^nG(\In(T_j))\to \underline{\scO}(T')\times 
\prod_{j=1}^{i-1}G(\In(T_j))\times \prod_{k=1}^rG(\In(T_{i_k}))\times
\prod_{j=i+1}^{n}G(\In(T_j))$$
which sends a decorated tree $T\in \underline{\scO}(T)$ to the decorated tree $T'$ obtained from $T$ as in the shrinking relation of Definition
\ref{rect4}(1b) disregarding lengths. On the other factors the map is given by identities and the map $\tau^\ast$ of Definition \ref{rect4}(1b) .

(2) Shrinking a nonbottom edge corresponds under $F^G$ to the map $(T;g_1,\ldots,g_n)\mapsto (T';g_1,\ldots,g_n)$ where $T'$ is obtained
from $T$ as in (1).

(3) Let $\tau:[T]\to[T']$ be a chopping morphism, represented by a tree $T$ of the form \ref{rect2a} with
 exactly one marked edge $e$ with marking $c$. This edge belongs to some subtree
$T_i$ of $T$, it could be its root. Then
$F^G(\tau)$ is induced by the map
$$
H: \underline{\scO}(T)\times \prod_{j=1}^nG(\In(T_j))\to \underline{\scO}(T')\times\prod_{j=1}^{i-1}G(\In(T_j))
\times G(\In(T'_i))\times \prod_{j=i+1}^{n}G(\In(T_j))$$
where $T'$ is obtained from $T$ and $T'_i$ from $T_i$ by deleting the subtrees with root edge $e$ (if $e$ is 
the root of $T_i$ then $T'_i$ is the trivial tree).
The map $H$ is given on $\underline{\scO}(T)\to \underline{\scO}(T')$ by the projection (the set of decorated nodes in $T'$ is a 
subset of the set of decorated nodes in $T$),
and on the other factors by the identities and the map $G(\widehat{c})$ of Definition \ref{rect4}(2).

In each case the equivariance relations on the operad $\scO$ and the functoriality of
$G$ imply that the definition of $F^G$ on the morphisms of $\scI$ does not depend on the
choice of representatives and that $F^G$ is a well-defined functor.
 \begin{leer}\label{diagram5}
 \textbf{A relative version:}
There is a relative version of this construction with respect to a map of operads
$\varphi:\scO\to\mathcal{P}$ in $\scS$.  Again let $G:\hscO\to \scS$ be an $\hscO$- diagram in $\scS$.  We then define the
functor $F^G_\varphi:\scI\to \scS$ in
exactly the same way as we defined $F^G$, except that for $F^G_\varphi[T]$ the bottom node of a representing decorated $T$ is decorated
with an element of $\mathcal{P}(k)$ instead of $\scO(k)$. Thus
$$F^G_\varphi([T])=\left(\coprod_{T\in[T]}\mathcal{P}(k)\times \prod_{i=1}^k W(G,T_i)\right)/\sim$$
with the equivariance relation as above. On morphisms 
$F^G_\varphi$ is defined in the same way as $F^G$, except that when we shrink
a bottom edge, we apply $\varphi$ to the element of $\scO$ decorating the node at the top of the
edge before we compose it with the element of $\mathcal{P}$ decorating the bottom node.
\end{leer}

\section{Homotopy colimits}\label{hocolim}

For a diagram $D:\scC\to \Cat$ in $\Cat$ the Grothendieck construction $\scC\int D$ is the category whose objects are 
pairs $(c,X)$ with $c\in \ob \scC$ and $X\in \ob D(c)$. A morphism $(c,X)\to (c',X')$ is a pair $(j,f)$ consisting of
a morphism $j:c\to c'$ in $\scC$ and a morphism $f:D(j)(X)\to X'$ in $D(c')$. Composition is the obvious one.

\begin{prop}\label{hocolim1} If $\scO$ is an operad in $\Cat$ and $G: \widehat{\scO}\to \Cat$ is an $ \widehat{\scO}$-category then
$\scI\int F^G$ is an $\scO$-algebra.
\end{prop}
\begin{proof}
 We define 
$$(*)\qquad\scO(m)\times_{\Sigma_m}\left(\scI\int F^G\right)^m\longrightarrow \scI\int F^G$$
as follows. A planar representative of an object on the left side of (*) looks like
\newline
$$
\left(A ,\left\{\left(T_i=
\begin{array}{c}
\begin{tikzpicture}
[level distance=10mm,
every node/.style={fill=white,circle,inner sep=0.5pt},
level 1/.style={sibling distance=10mm,nodes={fill=black}},
level 2/.style={sibling distance=10mm,nodes={fill=white}},]
\node[white] {}[grow=north]
child {node[draw] (above node) {}
child {node {$T_{ik_i}$}}
child [missing] {}
child {node {$T_{i2}$}}
child {node {$T_{i1}$}}
};
\draw[] (0.5,1.5) node [above] {$\cdots$};
\draw[] (0.35,0.5) node [above] {$B_i$};
\end{tikzpicture}
\end{array},\;\overline{C}_i
\right)\right\}^m_{i=1}\right)
$$
where $A$ is an object in $\scO(m)$, $T_i$ is a planar tree whose nodes are decorated by objects in the appropriate $\scO(k)$,
and $\overline{C}_i$ is an object in $G(\In(T_{i1}))\times\ldots\times G(\In(T_{ik_i}))$. The underlying tree of $T_i$ 
represents an object $[T_i]$ in $\scI$ and the pair $X_i=(T_i,\overline{C}_i)$ an object in $F^G([T_i])$.
We send this object to the object represented by
$$
\left(
\begin{array}{c}
\begin{tikzpicture}
[level distance=10mm,
every node/.style={fill=white,circle,inner sep=0.5pt},
level 1/.style={sibling distance=10mm,nodes={fill=black}},
level 2/.style={sibling distance=10mm,nodes={fill=white}},]
\draw[] (1.8,-0.85) node [above] {\scriptsize $A\circ(B_1\oplus B_2\oplus\cdots\oplus B_m)$};
\node[white] {}[grow=north]
child {node[draw] (above node) {}
child {node {$T_{m{k_m}}$}}
child [missing] {}
child {node {$T_{m1}$}}
child [missing] {}
child {node {$T_{1k_1}$}}
child [missing] {}
child {node {$T_{11}$}}
};
\draw[] (-2,1.7) node [above] {$\cdots$};
\draw[] (0,1.7) node [above] {$\cdots$};
\draw[] (2,1.7) node [above] {$\cdots$};
\end{tikzpicture}
\end{array},\quad(\overline{C}_1,\overline{C}_2,\ldots,\overline{C}_m)
\right)
$$
For later use we denote this representative by $\bar{\mu}(A ;X_1,\ldots ,X_m)$.
This map extends to morphisms: the $\scO(m)$ factor only affects $A$ while morphisms in $\scI\int F^G$ may result in operad compositions 
from the right of the $B_i$ with other node decorations of $T_i$, chopping or shrinking of internal edges and their affects on the $\overline{C}_i$.
\end{proof}

\begin{defi}\label{hocolim2}
 For a diagram $F: \mathcal{I}\to \Top $ of topological spaces we define $\hocolim_{\mathcal{I}} \ F$ to be the 2-sided
bar construction
$$\hocolim_{\mathcal{I}} \ F = B(\ast, \mathcal{I}, F)$$
where $\ast: \mathcal{I}^{\op}\to \Top$ is the constant diagram on a point (see \cite[3.1]{HV} for a list of properties of the 
the 2-sided bar construction). More explicitly, $B(\ast, \mathcal{I}, F)$
is the topological realization of the simplicial space
$$[n]\mapsto B_n(\ast, \mathcal{I}, F)=\coprod_{A,B} \mathcal{I}_n(A,B)\times F(A)$$
where $\mathcal{I}_n(A,B)\subset (\textrm{mor} \ \mathcal{I})^n$ is the subset of composable morphisms
$A\xrightarrow{f_1}\ldots \xrightarrow{f_n}B$. The degeneracy maps are defined as in the nerve of $\mathcal{I}$, the 
boundary maps $d^i:B_n(\ast, \mathcal{I}, F)\to B_{n-1}(\ast, \mathcal{I}, F)$ are defined as in the nerve for $i>0$,
while $d^0(f_1,\ldots,f_n;x))=(f_2,\ldots,f_n;F(f_1)(x))$.
\end{defi}

\begin{prop}\label{hocolim3} If $\scO$ is an operad in $\Top$ and $G: \widehat{\scO}\to \Top$ is an $ \widehat{\scO}$-space then
$\hocolim_{\scI}\ F^G$ is an $\scO$-algebra.
\end{prop}
\begin{proof}
 Since the classifying space functor preserves products it suffices to show that $B_\ast(\ast,\scI,F^G)$ is a simplicial object in
 the category of $\scO$-algebras. By Remark \ref{diagram2} an $m$-tuple of elements in $B_p(\ast,\scI,F^G)$ can be represented by
 sequences of marked trees
 $$\{(T_{j0}\xrightarrow{t_{j1}}\cdots \xrightarrow{t_{jp}} T_{jp};X_j)\}_{1\leq j\leq m}$$
 where $t_{jk}$ is $T_{jk}$ with a marking. An operation $a\in \scO(m)$ maps this $m$-tuple of elements to the element
 represented by
 $$(\mu(T_{10},\ldots,T_{n0})\xrightarrow{\mu(t_{11},\ldots,t_{n1})}\cdots \xrightarrow{\mu(t_{1p},\ldots,t_{np})} \mu(T_{1p},\ldots,T_{np}); 
 \bar{\mu}(a ;X_1,\ldots ,X_m))$$
 where $\mu(T_1,\ldots,T_n)$ is the tree obtained from $T_1,\ldots,T_n$ by gluing their roots together
 and $\mu(t_1,\ldots,t_n)$ the corresponding marked tree, while $\bar{\mu}
(a;X_1,\ldots,X_n)$ is defined as in the proof of Proposition \ref{hocolim1}.
\end{proof}

If $\scO$ is an operad in $\Cat$, then $B\scO$ is an operad in $\Top$ by Lemma \ref{operads6a}.

\begin{prop}\label{hocolim4} If $\scO$ is an operad in $\Cat$ and $G: \widehat{\scO}\to \Cat$ is an $ \widehat{\scO}$-category then
 $\hocolim_{\scT} \ B(F^G)$ is a $B\scO$-space.
\end{prop}
\begin{proof} By definition, $\hocolim_{\scT} \ B(F^G)$ is the topological realization of the bisimplicial set
 $$ 
([p],[q])\mapsto N_p(\scI)\times N_q(F^G([T_0]))
$$
where $N$ is the nerve functor.
An element in $N_p(\scI)\times N_q(F^G([T_0]))$ is a pair 
$$
([T_0]\xrightarrow{[t_1]}\cdots \xrightarrow{[t_p]} [T_p],
X_0\xrightarrow{x_1}\cdots \xrightarrow{x_q} X_q)
$$
 where $[T_0]\xrightarrow{t_1}\cdots \xrightarrow{t_p} [T_p]$ is a 
sequence of morphisms in $\scI$ and $X_0\xrightarrow{x_1}\cdots \xrightarrow{x_q} X_q$ is a sequence of morphisms in the
category $F^G([T_0])$.
We define an operation of $N_\ast\scO$ on its diagonal: The element $(A_0\xrightarrow{\alpha_1}\cdots \xrightarrow{\alpha_p} A_p)\in N_p\scO(n)$ maps 
the $n$-tuple represented by
$$\left\{\left( \begin{array}{c}
                       T_{j0}\xrightarrow{t_{j1}}\cdots \xrightarrow{t_{jp}} T_{jp}   \\
 X_{j0}\xrightarrow{x_{j1}}\cdots \xrightarrow{x_{jp}} X_{jp}
                          \end{array}
\right)\right\}_{1\leq j\leq n}
$$
to the object represented by the pair of sequences
$$
\left(\begin{array}{c}
       \mu(T_{10},\ldots,T_{n0})\xrightarrow{\mu(t_{11},\ldots,t_{n1})}\cdots \xrightarrow{\mu(t_{1p},\ldots,t_{np})} \mu(T_{1p},\ldots,T_{np})\\
\bar{\mu}(A_0;X_{10},\ldots,X_{n0})\xrightarrow{\bar{\mu}(\alpha_1;x_{11},\ldots,x_{n1})}\cdots 
\xrightarrow{\bar{\mu}(\alpha_p;x_{1p},\ldots,x_{np}\bar{\mu}}(A_p;X_{1p},\ldots,X_{np})
      \end{array}
\right)
$$
\end{proof}

In degree $p$ the nerve $N_\ast(\scI\int F^G)$ consists of diagrams
$$
([T_0],X_0)\xrightarrow{([t_1],x_1)}\cdots\xrightarrow{([t_p],x_p)} ([T_p],X_p)
$$
with $X_i\in F^G([T_i])$, $[t_i]:[T_{i-1}]\to [T_i]$ in $\scI$, and $x_i:[t_i](X_{i-1})\to X_i$ in $F^G([T_i])$. We always tacitly assume that
the representing trees $T_i$ and the marked trees $t_i$ are chosen as in Remark \ref{diagram2}. The 
$\scO$-structure on $\scI\int F^G$ defined in the proof of Proposition \ref{hocolim1} translates to an $N_\ast\scO$-structure on
$N_\ast(\scI\int F^G)$ as follows:
If $A_0\xrightarrow{\alpha_1}\cdots \xrightarrow{\alpha_p} A_p$ is
an element in $N_p\scO(n)$ it maps an $n$-tuple
$$
\left\{ ([T_{j0}],X_{j0})\xrightarrow{([t_{j1}],x_{j1})}\cdots\xrightarrow{([t_{jp}],x_{jp})} ([T_{jp}],X_{jp})\right\}_{1\leq j\leq n}
$$
to 
$$
([\mu(T_{10},\ldots,T_{n0})],\bar{\mu}(A_0;X_{10},\ldots,X_{n0}))\rightarrow\cdots\rightarrow ([\mu(T_{1p},\ldots,T_{np})],\bar{\mu}(A_p;X_{1p},\ldots,X_{np}))
$$
in the notation above with the obvious maps.

Thomason \cite{T} constructed a natural weak equivalence $\eta:\hocolim_\scI \ B(F^G)\to B(\scI\int F^G)$ defined on nerves by mapping 
$$
\left(\begin{array}{c}
                          [T_0]\xrightarrow{[t_1]}\cdots \xrightarrow{[t_p]} [T_p]\\
 X_0\xrightarrow{x_1}\cdots \xrightarrow{x_p} X_p
                          \end{array}
\right)
$$
to
$$
([T_0],X_0)\xrightarrow{([t_1],[t_1](x_1))}\cdots \xrightarrow{([t_p],[t_p]\circ \ldots\circ [t_1](x_p))} ([T_p],[t_p]\circ \ldots\circ [t_1]X_p)).
$$

\begin{prop}\label{hocolim5} $\eta:\hocolim_{\scI} \ B(F^G)\to B(\scI\int F^G)$ is a weak equivalence of $B\scO$-spaces natural in $G$.
\end{prop}
\begin{proof}
We prove this on the level of nerves. So let
$$\overline{(T,X)}=\left\{\left( \begin{array}{c}
                       [T_{j0}]\xrightarrow{[t_{j1}]}\cdots \xrightarrow{[t_{jp}]} [T_{jp}]   \\
 X_{j0}\xrightarrow{x_{j1}}\cdots \xrightarrow{x_{jp}} X_{jp}
                          \end{array}
\right)\right\}_{1\leq j\leq n}
$$
be an element in $\prod_{j=1}^n N_p(\scI)\times N_p(F^G([T_{j0}]))$ and 
$\bar{A}=(A_0\xrightarrow{\alpha_1}\cdots \xrightarrow{\alpha_p} A_p)\in N_p\scO(n)$, we have to show that
$\bar{A}\ast \eta^n(\overline{(T,X)})=\eta(\bar{A}\ast \overline{(T,X)})$, where $\bar{A}\ast -$ stands for the 
operation of $\bar{A}$.

To avoid a multitude of indices we do this for $n=2$ and $p=1$; the general case is analogous. Then
$$
\eta^2(\overline{(T,X)})=\{([T_{i0}],X_{j0})\xrightarrow{([t_{j1}],[t_{j1}](x_{j1}))} ([T_{j1}],[t_{j1}](X_{j1}))\}_{j=1,2}
$$
and
$$
\bar{A}\ast \eta^2(\overline{(T,X)})= ([\mu(T_{10},T_{20})],\bar{\mu}(A_0;X_{10},X_{20}))\to ([\mu(T_{11},T_{21})],
\bar{\mu}(A_1;[t_{11}](X_{11}),[t_{21}](X_{21}))
$$
Now 
$$
\bar{A}\ast \overline{(T,X)}= \left\{\left( \begin{array}{c}
                       [\mu(T_{10},T_{20})]\xrightarrow{[\mu(t_{11},t_{21})]} [\mu(T_{11},T_{21})]   \\
 \bar{\mu}(A_0;X_{10},X_{20})\xrightarrow{\bar{\mu}(\alpha_1;x_{11},x_{21})}  \bar{\mu}(A_1;X_{11},X_{21})
                          \end{array}
\right)\right\},
$$
which is mapped by $\eta$ to
$$
([\mu(T_{10},T_{20})], \bar{\mu}(A_0;X_{10},X_{20}))\rightarrow 
([\mu(T_{11},T_{21})], [\mu(t_{11},t_{21})] (\bar{\mu}(A_1;X_{11},X_{21}))).
$$
So we have to show that
$$
\begin{array}{rcl}
[\mu(t_{11},t_{21})](\bar{\mu}(A_1;X_{11},X_{21})) & = & \bar{\mu}(A_1;[t_{11}](X_{11})),[t_{21}](X_{21})) \\
{[\mu(t_{11},t_{21})]} (\bar{\mu}(\alpha_1 ;x_{11},x_{21})) & = & \bar{\mu} (\alpha_1;[t_{11}](x_{11}),
 [t_{21}](x_{21}))\\
\end{array} 
$$
These equations hold because the operation of $\bar{A}$ is defined by composition from the left with the sum of the appropriate root labels, while 
the $t_{ij}$ shrink edges of trees or are evaluations which could at most result in compositions with root labels from the right. 
\end{proof}

\begin{prop}\label{hocolim6} If $\scO$ is $\Sigma$-free there is a natural homeomorphism $B(F^G)([T])\cong F^{BG}([T])$
and hence a homeomorphism
 $$\hocolim_{\scI} \ B(F^G) \cong  \hocolim_{\scI}\ F^{BG}$$
of $B\scO$-spaces natural in $G$. In particular, Thomason's map induces a weak equivalence of $B\scO$-spaces
$\hocolim_{\scI}\ F^{BG}\to B(\scI\int F^G)$.
\end{prop}
\begin{proof}
$F^G([T])\cong \underline{\scO}(T)\times_{\Aut(T)}\lambda_G(T)$ and $F^{BG}([T])\cong \underline{B\scO}(T)\times_{\Aut(T)}\lambda_{BG(T)}$
by \ref{free_3}. Since $B$ is product preserving there is a natural homeomorphism $\lambda_{BG}\to B\lambda_G$. Since $\Aut(T)$
acts freely on $\underline{\scO}(T)$ it acts freely on $\underline{\scO}(T)\times \lambda_G(T)$. 
Hence there is a natural homeomorphism $B(\underline{\scO}(T)\times_{\Aut(T)}\lambda_G(T))\to B\underline{\scO}(T)\times_{\Aut(T)}B\lambda_G(T)$
by Lemma \ref{operads6}. Here we also use that $B(\underline{\scO}(T)\times \lambda_G(T))\cong B(\underline{\scO}(T))\times B(\lambda_G(T))$.
\end{proof}

A map $\tau:G_1\rightarrow G_2$ of $\hscO$-categories induces a map
$F^\tau:F^{G_1}\to F^{G_2}$ of $\scI$-diagrams in $\Cat$ and hence a map 
$$\scI\int F^\tau:\scI\int F^{G_1}\to \scI\int F^{G_2}$$
of $\scO$-algebras.

\begin{prop}\label{hocolim8}
 If $\scO$ is a $\Sigma$-free operad in $\scS$ and $\tau:G_1\rightarrow G_2$ is a map of $\hscO$-diagrams in $\scS$
which is objectwise a weak equivalence, then 
\begin{enumerate}
 \item if $\scS=\Cat$ the functor 
$\scI\int F^\tau: \scI\int F^{G_1}\to  \scI\int F^{G_2}$ is a weak equivalence of $\scO$-algebras.
\item if $\scS=\Top$, the map $\hocolim_{\scI}\ F^{G_1}\to \hocolim_{\scI}\ F^{G_2}$ is a weak equivalence of $\scO$-spaces.
\end{enumerate}
\end{prop}
\begin{proof}
By Proposition \ref{hocolim6} it suffices to prove (2) because weak equivalences in $\Cat$ are detected by the classifying space functor $B$.

We have a commutative diagram 
$$\xymatrix{
F^{G_1}([T])\ar[rr]^{F^{\tau}([T])}\ar[d]^\cong && F^{G_2}([T])\ar[d]^\cong\\
\underline{\scO}(T)\times_{\Aut(T)}\lambda_{G_1}(T)\ar[rr]^{\id\times_{\Aut} \lambda_\tau(T)} && 
\underline{\scO}(T)\times_{\Aut(T)}\lambda_{G_2}(T)
}
$$
Since $\underline{\scO}(T)$ is a numerable principal $\Aut(T)$ space by Lemma \ref{rect_6b} and $\lambda_\tau(T)$
is a homotopy equivalence, the map $\id\times_{\Aut} \lambda_\tau(T)$ is a homotopy equivalence by \ref{rect_6a}(2).
Hence $F^\tau :F^{G_1}\to F^{G_2}$ is objectwise a weak equivalence, inducing a homotopy
equivalence $\hocolim \ F^{G_1}\to \hocolim \ F^{G_2}$.
\end{proof}

\section{Change of operads}\label{change}

Let $X$ be an $\scO$-algebra in $\Cat$ and $\widehat{X}: \hscO\to \Cat$ its associated $\hscO$-diagram.
Then we have a map
$$\varepsilon: \scI\int F^{\widehat{X}}\to X
$$
induced by
 $$\xymatrix{
\scI \ar[rr]^p\ar[rdd]_{F^{\widehat{X}}} && \ast\ar[ldd]^X\\
& \stackrel{e}{\Rightarrow} & \\
&\Cat &
}
$$

where $e([T])$ is the composite
$$e([T]): F^{\widehat{X}}([T])\rightarrow \scO(\In(T))\times_{\Sigma_{\In(T)}}X^{\In(T)}\rightarrow X.$$
If $T$ is of the form \ref{rect2a} then the first map
shrinks all edges of all decorated trees $T_i$. The second map is the $\scO$-action on $X$. 

By construction, $\varepsilon$ is a homomorphism of $\scO$-algebras.

\begin{prop}\label{change1}
 The homomorphism $\varepsilon: \scI\int F^{\widehat{X}}\to X$ is a weak equivalence of $\scO$-algebras.
\end{prop}

\begin{proof}
Note that $F^{\widehat{X}}([T])= W(\widehat{X},[T])$ modulo the equivariance relation, because $\widehat{X}(n)=X^n$. 
There is a section $s: X\to \scI\int F^{\widehat{X}}$ of  $\varepsilon$,
 which is not a map of algebras. It is induced by
$$\xymatrix{
\ast \ar[rr]^i\ar[rdd]_X && \scI \ar[ldd]^{F^{\widehat{X}}}
\\
& \stackrel{\iota}{\Rightarrow} & \\
& \Cat &
}
$$
where $i$ takes $\ast$ to the tree $\Theta_1$ which in turn
 is mapped to $\scO(1)\times X$ by  $F^{\widehat{X}}$, and $\iota(\ast):X\to F^{\widehat{X}}\circ i$ is the inclusion
$X=\{\id \}\times X\subset \scO(1)\times X$. 

Let $j:\scI\to\scI$ be the functor which maps $[T]$ to the isomorphism class represented by
$$
\begin{tikzpicture}
[level distance=10mm,
every node/.style={fill=white,circle,inner sep=0.5pt},
level 1/.style={sibling distance=10mm,nodes={fill=white}},]
\node[white] {}[grow=north]
child {node[draw] (above node) {}
child {node {$T$}
}
};
\end{tikzpicture}
$$
Now let $J: \scI\int F^{\widehat{X}}\to \scI\int F^{\widehat{X}}$ be the functor sending an object
$([T],X)$ with $[T]\in \scI$ and $X\in F^{\widehat{X}}([T])$ to $(j([T]),j(X))$, where $j(X)$ 
has the decoration of $X$ on the $T$-part and $\id$ as decoration of the root of
$j([T])$. This definition extends canonically to morphisms with the root of $j([T])$ always 
decorated by the identity.

Shrinking and chopping the incoming edge of the root of $j([T])$ defines natural transformations
$J\Rightarrow \Id$ respectively $J\Rightarrow s\circ \varepsilon$. The classifying space functor
turns these transformations into homotopies 
$\Id\simeq BJ\simeq Bs\circ B\varepsilon$.
\end{proof}

\begin{coro}\label{change2}
 Let $\varphi: \scO\to \scP$ be a weak equivalence of $\Sigma$-free operads and let $X$ be an $\mathcal{O}$-algebra.
Then there are natural weak equivalences of $\mathcal{O}$-algebras
$$
X\leftarrow \scI\int F^{\widehat{X}} \rightarrow \scI\int F^{\widehat{X}}_\varphi .
$$
In particular, $X$ is weakly equivalent to an $\scP$-algebra.
\end{coro}
\begin{proof}
 The left map is a weak equivalence of $\mathcal{O}$-algebras by Proposition \ref{change1}, and the right map is a weak equivalence 
of $\mathcal{O}$-algebras since $F^{\widehat{X}} \rightarrow F^{\widehat{X}}_\varphi$ is objectwise a weak equivalence.
\end{proof}

The analogous results hold in $\Top$: by \cite[Proposition 3.1]{HV} the functors and natural transformations constructed
in the proof of Proposition \ref{change1} imply

\begin{prop}\label{change4}
 (1) Let $\scO$ be a $\Sigma$-free topological operad, $X$ an $\scO$-space and $\widehat{X}: \hscO\to \Top$ its associated 
$\widehat{\scO}$-diagram, then there is a natural homomorphism of $\scO$-spaces 
$\varepsilon:\hocolim \ F^{\widehat{X}}\to X$, which is a weak equivalence.\\
(2) If $\varphi: \mathcal{O}\to \scP$ is a weak equivalence of $\Sigma$-free topological operads, then there are
natural weak equivalences of $\mathcal{O}$-spaces
$$
X\leftarrow \hocolim \ F^{\widehat{X}} \rightarrow \hocolim \ F^{\widehat{X}}_\varphi .
$$
In particular, $X$ is weakly equivalent to an $\scP$-space.
\end{prop}

\section{Comparing $\widehat{\scO}$-categories $G$ with $\scI\int F^G$}\label{Ohat}

Let $G:\hscO\to\Cat$ be an $\scO$-category. We define a functor
$$\lambda_n=\lambda_{nG}:\tmT^n\to \Cat$$
by sending $(T_1,\ldots ,T_n)$ to $G(\In(T_1)+\ldots +\In(T_n))$ and a morphism $(\phi_1,\ldots ,\phi_n)$ to
$G(\phi_1^\Sigma\oplus\ldots\oplus \phi_n^\Sigma)$. Let
$$
W^G_n:\scI^n\to\Cat
$$
be the diagram given on objects by the coend
$$
W^G_n([T_1],\ldots,[T_n])=\underline{\scO}\times\ldots\times\underline{\scO}\otimes_{[T_1]\times\ldots\times [T_n]}\lambda_n
$$
obtained by restricting the functors $\lambda_n$ and
$$
\underline{\scO}\times\ldots\times\underline{\scO}:\tmT^n\to \Cat,\quad (T_1,\ldots,T_n)\mapsto
\underline{\scO}(T_1)\times\ldots\times\underline{\scO}(T_n)
$$
to $[T_1]\times\ldots\times[T_n]\subset\tmT^n$. On generating morphisms $W^G_n$ is defined as follows:
\begin{itemize}
	\item Shrinking an internal edge: $W^G_n(-)$ is defined in the same way as $F^G(-)$ for shrinking a nonbottom edge
	\item Chopping of a branch: $W^G_n(-)$ is defined in the same way as $F^G(-)$ with the difference that $G(\widehat{c})$ is defined with
	respect to the union of all inputs of $T_1,\ldots,T_n$.
\end{itemize}
As before, an object or morphism of $W^G_n([T_1],\ldots,[T_n])$ is represented by a tuple $(T_1,\ldots,T_n,C)$ consisting of trees $T_i\in [T_i]$
decorated by objects respectively morphisms in $\scO(\In(T_i))$ and an object respectively morphism 
$C\in G(\In(T_1)+\ldots +\In(T_n))$. The appropriate equivariance relations hold.
\begin{rema}\label{Ohat0}
Note that, unlike in the $F^G$ construction, the bottom edges of trees play no special role
in the $W^G_n$ construction. Also note that $W^G_1([T])$, modulo the equivariance relation, coincides with the construction
$W(G,T)$ used as a stepping stone for the $F^G$ construction. 
\end{rema}

\begin{lem}\label{Ohat1} The correspondence
$$
n\mapsto \scI^n\int W^G_n
$$
extends to an $\hscO$-category
$$
M=\scI^\ast\int W^G_\ast:\hscO\to\Cat.
$$
\end{lem}

\begin{proof}
If $\sigma\in\Inj(k,l)$ then $M(\sigma^\ast):M(l)\to M(k)$ is induced by the projection
$$
\scI^l\to\scI^k, \quad ([T_1],\ldots,[T_l])\mapsto([T_{\sigma(1)}],\ldots,[T_{\sigma(k)}])
$$
and the map
$G(\sigma(\In(T_1),\ldots,\sigma(\In(T_l))^\ast)$  (see \ref{operads4}).
If $\alpha=(\alpha_1,\ldots,\alpha_r)\in\scO(k_1,1)\times\ldots\times\scO(k_r,1)$ and $m=k_1+\ldots,k_r$, 
then $M(\alpha):M(m)\longrightarrow M(r)$ maps a representing tuple $(T_1,\ldots,T_m,C)$, where each $T_i$ 
is a decorated tree and $C\in G(\In(T_1+\ldots+\In(T_m))$, to the element represented 
by $(T'_i,\ldots,T'_r,C)$. Here $T'_i$ is obtained by grafting the roots of $T_{p+1},\ldots,T_{p+k_i}$ together, 
where $p=k_1+\ldots+k_{i-1}$, and decorating the root of $T'_i$ by $\alpha_i\circ(\beta_1\oplus\ldots\oplus\beta_{k_i})$
if $\beta_j$ is the root node decoration of $T_{p+j}$.
\end{proof}

\begin{lem}\label{Ohat2} 
If $\scO$ is a $\Sigma$-free operad, there is a map of $\hscO$-categories
$$
\tau :M=\scI^\ast\int W^G_\ast\longrightarrow \widehat{\scI\int F^G},
$$
natural in $G$, which is objectwise a weak equivalence if $G$ is special.
\end{lem}

\begin{proof}
$\tau(n):M(n)\longrightarrow(\scI\int F^G)^n$ sends a representing element $(T_1,\ldots,T_n,C)$ to 
$((T_1,G(\sigma^\ast_1)(C)),\ldots,(T_n,G(\sigma^\ast_n)(C))$, where $\sigma_i:\underline{\In(T_i)}\to\underline{\In(T_1)+\ldots+\In(T_n)}$ is the canonical inclusion. 

By construction, this defines a map of $\hscO$-categories.

If $G$ is special the map
$$
(G(\sigma^\ast_1),\ldots,G(\sigma^\ast_n)):G(\In(T_1)+\ldots+\In(T_n))\rightarrow \prod^n_{i=1} G(\In(T_i))
$$
is a weak equivalence. Consequently, $\tau$ is a weak equivalence because $\scO$ is $\Sigma$-free (see Lemma \ref{operads6} and
also \ref{rect_6a}(2)).
\end{proof}

\begin{lem}\label{Ohat3} 
There is a map of $\hscO$-categories
$$
\varepsilon:M=\scI^\ast\int W^G_\ast\longrightarrow G,
$$
natural in $G$, which is objectwise a weak equivalence.
\end{lem}

\begin{proof}
$\varepsilon(n):M(n)\to G(n)$ is defined on $(T_1,\ldots,T_n,C)$ by chopping off the roots of $T_1,\ldots,T_n$ as
 explained in the definition of $F^G$. 
To prove that $\varepsilon(n)$ is a weak equivalence we proceed as in the proof of Proposition \ref{change1}.
The functor $\varepsilon(n)$ has a section $s_n:G(n)\to M(n)$ sending $C\in G(n)$ to $(\Theta_1,\ldots,\Theta_1,C)\in M(n)$, 
where the node of $\Theta_1$ is decorated by the identity.

We define a functor $J:\scI^n\int W^G_n\to\scI^n\int W^G_n$ in the same way as in Proposition \ref{change1}, we map $([T_1],\ldots,[T_n],X)$ 
to $(j[T_1],\ldots,j[T_n];j(X))$ 
with the difference that $j(X)$ is obtained from $X$ by decorating each of the $n$ new root nodes by the identity. By shrinking 
and chopping the incoming edges 
to the root nodes we again obtain natural transformations $J\Rightarrow\Id$ and $J\Rightarrow s_n\circ\varepsilon(n)$.
\end{proof}

Combining the preceding three lemmas we obtain:

\begin{theo}\label{Ohat4}
Let $\scO$ be a $\Sigma$-free operad in $\Cat$. Then there are functors
$$
\scI\int F^{(-)}:\Cat^{\widehat{\scO}}\to \scO\mbox{-}\Cat\quad\textrm{ and }\quad
\scI^\ast\int W^{(-)}_\ast: \Cat^{\widehat{\scO}}\to\Cat^{\widehat{\scO}}
$$
and natural transformations of functors $\Cat^{\widehat{\scO}}\to\Cat^{\widehat{\scO}}$
$$
\widehat{\scI\int F^{(-)}}\stackrel{\tau}{\longleftarrow} \scI^\ast\int W^{(-)}_\ast \stackrel{\varepsilon}{\longrightarrow} \Id
$$
such that 
\begin{enumerate}
	\item each $\varepsilon(G)(n)$ is a weak equivalence
	\item if $G$ is a special, then each $\tau(G)(n) $ is a weak equivalence.\hfill\ensuremath{\Box}
\end{enumerate}
\end{theo}

\section{Comparing $\hat{\scO}$-spaces $G$ with $\hocolim_{\scI}\ F^G$}\label{recttop}

As  one would expect, there are topological versions of the constructions and results of Section \ref{Ohat}. We will give a short account of these.
In the process we will use the following properties of the homotopy colimit construction (e.g. see \cite[Prop. 3.1]{HV}):

\begin{leer}\label{tecttop1}
 Let $\scC$ and $\scD$ be small categories, $F,\ G:\scC\to \scD$ functors, $X:\scC\to \Top$ and $Y:\scD\to \Top$ diagrams,
 $$\tau: F\to G, \quad \alpha:X\to Y\circ F,\quad \beta:X\to Y\circ G$$
 natural transformations such that $(Y\star \tau)\circ \alpha = \beta$.
 $$\xymatrix{
 \scC \ar[rr]^{F,G}\ar[rd]_X && \scD\ar[ld]^Y &\quad & & X\ar[ld]_\alpha \ar[rd]^\beta &\\
 & \Top &&& Y\circ F\ar[rr]^{Y\star\tau} && Y\circ G
 }
 $$
 Then $F$ and $\alpha$ induce a map
 $$B(\ast,F,\alpha):\hocolim_{\scC}\ X\to \hocolim_{\scD}\ Y$$
 while $\tau$ defines a homotopy
 $$B(\ast,F,\alpha)\simeq B(\ast,G,\beta)$$
 (see Definition \ref{hocolim2} and \cite[3.1]{HV}).
\end{leer}

Let $\scO$ be an operad in $\Top$ and $G:\hscO\to \Top$ an
$\hscO$-space. Let
$$
V_n^G:\scI ^n\to \Top
$$
be the topological version of the diagram $W_n^G$, defined in exactly the same way. For a morphism $f\in \hscO(l,k)$ the corresponding map described in the proof
of Lemma \ref{Ohat1} defines a functor $\scI(f):\scI^l\to \scI^k$ together with a natural transformation 
$$\overline{f}:V_l^G\to V_k^G\circ \scI(f)$$
which induces a map
$$
\hocolim_{\scI^l}\ V_l^G\to \hocolim_{\scI^k}\ V_k^G,
$$
and we obtain an $\hscO$-space
$$
\hocolim_{\scI^\ast}\ V_\ast ^G: \hscO\to \Top,\quad n\mapsto \hocolim_{\scI^n}\ V_n^G.
$$

\begin{theo}\label{recttop2}
 If $\scO$ is a $\Sigma$-free operad in $\Top$. Then there are functors
 $$
 \hocolim_{\scI}\ F^{(-)}: \Top^{\hscO}\to \scO\mbox{-}\Top \quad \textrm{ and }\quad 
 \hocolim_{\scI^\ast}\ V_\ast ^{(-)}:\Top^{\hscO}\to\Top^{\hscO}
 $$
 and natural transformations of functors $\Top^{\hscO}\to\Top^{\hscO}$
$$
\widehat{\hocolim_{\scI}\ F^{(-)}}\stackrel{\tau}{\longleftarrow}  \hocolim_{\scI^\ast}\ V_\ast ^{(-)}  \stackrel{\varepsilon}{\longrightarrow} \Id
$$
such that 
\begin{enumerate}
	\item each $\varepsilon(G)(n)$ is a weak equivalence
	\item if $G$ is a special, each $\tau(G)(n)$ is a weak equivalence.
\end{enumerate}
\end{theo}
\begin{proof}
 The maps $\tau_n$, defined on representatives like the maps $\tau(n)$ in the proof of Lemma \ref{Ohat2}, define a map from the diagram $V_n^G$ to the diagram
 $$(F^G,\ldots,F^G):\scI^n\to \Top.$$
 Since topological realization preserves products it induces a map 
 $$\tau_n:\hocolim_{\scI^n}\ V_n^G \to (\hocolim_{\scI}\ F^G)^n.$$
 By construction, the $\tau_n$ define a map of $\hscO$-spaces. If $G$ is special, the map of diagrams 
 $V_n^G\to (F^G,\ldots,F^G)$ is objectwise a homotopy equivalence inducing a homotopy equivalence $\tau_n$.
 
 Let $\ast$ stand for the category with one object $0$ and the identity morphism, let $G(n):\ast \to \Top$ be the functor sending $0$ to $G(n)$, and let
 $P_n:\scI^n\to \ast$ be the projection. The maps $\varepsilon_n$ of the proof of Lemma \ref{Ohat3} define natural transformations $\beta_n: V_n^G\to G(n)\circ P_n$, thus inducing  maps
 $$\varepsilon_n=B(\ast,P_n,\beta_n): \hocolim_{\scI^n}\ V_n^G\to G(n)$$
 which define a map of $\hscO$-spaces.

Let $j:\scI\to \scI$ be the functor defined in the proof of Proposition \ref{change1}. Let $S:\ast\to \scI^n$ send $0$ to $([\Theta_1],\ldots,[\Theta_1])$ and let $\gamma: G(n)\to V^G_n\circ S$ be
the natural transformation sending $C\in G(n)$ to the element represented by $(\Theta_1,\ldots, \Theta_1,C)$ where the root nodes are decorated by identities. The pair $(S,\gamma)$ induces a section
$$
s=B(\ast,S,\gamma): G(n)\to \hocolim_{\scI^n}\ V_n^G$$ of $\varepsilon_n.
$
 The functor $j^n:\scI^n\to \scI^n$ together with the natural transformation $\alpha :V_n^G\to V_n^G\circ j^n$ sending a representative
$(T_1,\ldots,T_n,C)$ to $(j(T_1),\ldots,j(T_n),C)$, where the added root nodes are decorated by identities, define a selfmap
$$J:\hocolim_{\scI^n}\ V_n^G\to \hocolim_{\scI^n}\ V_n^G.$$
There are natural transformations $sh:j^n\to \Id$ and $ch:j^n\to S\circ P_n$ defined by shrinking respectively chopping the incoming edges to the root nodes. Since
$$ (V^G_n\star sh)\circ \alpha= \id\quad \textrm{and} \quad (V^G_n\star ch)\circ \alpha = \gamma\circ \beta_n$$
there are homotopies $J\simeq id$ and $J\simeq S\circ P_n$.
\end{proof}

\begin{coro}\label{recttop3}
 Let $\scO$ be a $\Sigma$-free operad in $\Top$. Then there is a chain of natural transformations
of functors $ \Top^{\hscO}\to \scO\mbox{-}\Top $ connecting the functor $M$ of Section \ref{simpleversion}
and the functor $\hocolim_{\scI}\ F^{(-)}$, which are weak equivalences when evaluated at 
special $\hscO$-spaces.
\end{coro}
 \begin{proof}
 We apply the rectification of Section \ref{simpleversion} to the diagram in Theorem \ref{recttop2} to obtain a diagram of weakly equivalent $\scO$-algebras
 $$
  M(G)\leftarrow M(\hocolim_{\scI^\ast}\ V_\ast ^G)\rightarrow M(\widehat{\hocolim_{\scI}\ F^G}).
 $$
 By Proposition \ref{rect7} the $\scO$-algebras $ M(\widehat{\hocolim_{\scI}\ F^G})$ and $\hocolim_{\scI}\ F^G$  are weakly equivalent.
\end{proof}
\section{From simplicial $\scO$-algebras to $\scO$-algebras}\label{simplicial}

Given a simplicial category $\scC_\ast:\Delta^{op}\longrightarrow \Cat$ there are the Bousfield-Kan map and the Thomason map

\begin{leer}\label{simpl1}
 $$
 |B(\scC_\ast)|\xleftarrow{\rho} \hocolim_{\Delta^{op}}\ B(\scC_\ast) \xrightarrow{\eta} B(\Delta^{\op}\int \scC_\ast)
 $$
\end{leer}
which are natural maps known to be homotopy equivalences by \cite[XII.3.4]{BK} or \cite[Theorem 18.7.4]{Hirsch}
respectively \cite[Theorem 1.2]{T}. So Thomason's homotopy colimit construction
in $\Cat$ replaces a simplicial category by a category in a nice way: their realizations in $\Top$ via the classifying space functor
are homotopy equivalent.

In this section we want to lift this result to simplicial $\scO$-algebras over a $\Sigma$-free operad $\scO$ in $\Cat$.

We start with the right map in \ref{simpl1} for which there is a more general version. Let $\scL$ be a small
indexing category. Let $\scO$ be an operad in $\scS$, where $\scS$ is $\Cat$ or $\Top$, and let $\scO\mbox{-}\scS^{\scL}$ denote the category
of $\scL$-diagrams of $\scO$-algebras in $\scS$. We have functors
$$\begin{array}{rcll}
 H_{\top}:\scO\mbox{-}\Top^{\scL}& \to &   \scO\mbox{-}\Top & \quad \textrm{if $\scO$ is an operad in $\Top$,}\\
 H_{\cat}:\scO\mbox{-}\Cat^{\scL} & \to & \scO\mbox{-}\Cat  & \quad \textrm{if $\scO$ is an operad in $\Cat$,}
  \end{array}
  $$
defined by $H_{\top}(X)=\hocolim_{\scI} F^{\hocolim_{\scL} \widehat{X}}$ for an $\scL$-diagram $X:\scL\to \scO\mbox{-}\Top$
and $H_{\cat}(D)=\scI\int F^{\scL\int \widehat{D}}$ for an $\scL$-diagram $D:\scL\to \scO\mbox{-}\Cat$, where 
$\widehat{X}:\widehat{\scO}\to \Top^{\scL}$ and $\widehat{D}:\widehat{\scO}\to \Cat^{\scL}$ are induced by $X$ and $D$ respectively.

\begin{prop}\label{simpl2} Let $\scO$ be a $\Sigma$-free operad in $\Cat$. Then there is a natural
weak equivalence
 $$
\eta: H_{\top}\circ B^{\scL}\Rightarrow B\circ H_{\cat}.
$$
of functors $\scO\mbox{-}\Cat^{\scL}\to B\scO\mbox{-}\Top$.
\end{prop}
 
\begin{proof} Let  $ D :\scL\to \scO\mbox{-}\Cat$ be an $\scL$-diagram of $\scO$-algebras. We have natural maps of $B\scO$-spaces
$$
\begin{array}{rcl}
H_{\top}(BD)=\hocolim_{\scI} F^{\hocolim_{\scL} \widehat{BD}} & \to &
\hocolim_{\scI} F^{B(\scL\int \widehat{D})}\cong \hocolim_{\scI}B(F^{\scL\int \widehat{D}})\\
 & \to & B(\scI\int F^{\scL\int \widehat{D}})= B(H_\cat(D)).
\end{array}
$$
By \cite[Theorem 1.2]{T} Thomason's map defines a pointwise weak equivalence 
$$\hocolim_{\scL} \widehat{BD}\to B(\scL\int \widehat{D}),$$
so the first map is a weak equivalence by Proposition \ref{hocolim8}. 
For the isomorphism see Proposition \ref{hocolim6},
and the second map is a weak equivalence by Proposition \ref{hocolim5}.
\end{proof}
In general we cannot say much about the homotopy type of $H_{\top}(X)$ and $H_{\cat}(D)$. This is different if $\scL=\Delta^{\op}$ and $X$ is proper.
Here we call a simplicial space \textit{proper} if the inclusions $sX_n\subset X_n$ of the subspaces $sX_n$ of the degenerate elements of $X_n$ are
closed cofibrations for all $n$, and we call a simplicial $\scO$-space proper if its underlying space is proper.

\begin{prop}\label{simpl3} Let $\scO$ be a $\Sigma$-free operad in $\Top$ and let $X_\ast$ be a proper simplicial $\scO$-space. Then there
 is a weak equivalence of $\scO$-spaces
 $$\rho: H_{\top}(X_\ast)\to |X_\ast|$$
 natural with respect to homomorphisms of proper simplicial $\scO$-spaces.
\end{prop}
\begin{proof}
 We have natural maps of $\scO$-spaces
 $$
 H_{\top}(X_\ast)=\hocolim_{\scI} F^{\hocolim_{\Delta^{\op}}\widehat{X_\ast}}\to \hocolim_{\scI}F^{|\widehat{X_\ast}|}\to |X_\ast|.
 $$
 The Bousfield-Kan map $\hocolim_{\Delta^{\op}}\widehat{X_\ast}\to |\widehat{X_\ast}|$ is pointwise a weak equivalence provided $X_\ast$ is proper,
 so that the first map is a weak equivalence by Proposition \ref{hocolim8}. Since topological realization preserves colimits 
 and finite products, we have a natural isomorphism
$
|\widehat{X_\ast}|\cong \widehat{|X_\ast|}
$, and the second map is a weak equivalence by Proposition \ref{change4}.
\end{proof}

Combining these results we obtain the passage from simplicial $\scO$-algebras in $\Cat$ to $\scO$-algebras in $\Cat$.

\begin{theo}\label{simpl4} Let $\scO$ be a $\Sigma$-free $\Cat$-operad, and let $B\scO\mbox{-}p\Top^{\Delta^{\op}}$ be the full subcategory of $B\scO\mbox{-}\Top^{\Delta^{\op}}$
of proper simplicial $B\scO$-spaces. Then there is a diagram
 $$\xymatrix{
 \scO\mbox{-}\Cat^{\Delta^{\op}}\ar[dd]_{H_{\cat}}\ar[rr]^{B^{\Delta^{\op}}}&& B\scO\mbox{-}p\Top^{\Delta^{\op}}\ar[dd]_{H_{\top}}^{\qquad\Rightarrow}  &\subset & B\scO\mbox{-}\Top^{\Delta^{\op}}
\ar[lldd]^{|-|}\\
 &\Leftarrow & &  \\
 \scO\mbox{-}\Cat\ar[rr]^B && B\scO\mbox{-}\Top
 }
 $$
commuting up to natural weak equivalences. 
 \end{theo}

Let $\const: \Cat \to \Cat^{\Delta^{\op}}$ be the constant simplicial object functor and let $\scC$ be an $\scO$-algebra.
Since $|B(\const\ \scC)|\cong  B(\scC)$ the functor $H\circ \const$ preserves the homotopy type. But we can do better. The
diagram
$$
\xymatrix{
\widehat{\scO} \ar[rr]^{\widehat{\scC}}\ar[rrd]_{\widehat{\const\ \scC}} && \Cat \ar[rrd]^{\Delta^{\op}\times -}\ar[d]^{\const}\\
&& \Cat{\Delta^{\op}}\ar[rr]^{\Delta^{\op}\int -} && \Cat
}
$$
commutes. 

\begin{prop}\label{simpl5}
 Let $\scO$ be a $\Sigma$-free $\Cat$-operad and let $\scC$ be an $\scO$-algebra. Then there are weak equivalences of
 $\scO$-algebras
 $$
 H(\const\ \scC)=\scI\int F^{\Delta^{\op}\int \widehat{\const\ \scC}}=\scI\int F^{\Delta^{\op}\times \widehat{\scC}}
 \xrightarrow{\pi} \scI\int F^{\widehat{C}}\xrightarrow{\varepsilon} \scC
 $$
 where $\pi$ is induced by the projection $\Delta^{\op}\times \widehat{\scC}\to \widehat{\scC}$ and $\varepsilon$ is the
 homomorphism of Proposition \ref{change1}.
\end{prop}
\begin{proof} This follows from Propositions \ref{change1} and \ref{hocolim8}, because the projection $\Delta^{\op}\times \widehat{\scC}\to \widehat{\scC}$
 is objectwise a weak equivalence.
\end{proof}

\begin{rema}\label{simpl6}
Our passage from simplicial algebras to algebras translates verbatim to $\Top$, but, of course, topological realization is
the preferred passage: it is well known that the topological realization of a simplicial $\scO$-space is an $\scO$-space in a canonical way.
\end{rema}
\section{An application}\label{application}
\begin{defi}\label{appli1}
 Let $\scO$ be an operad in $\Cat$ and let $\scP$ be an operad in $\Top$. In this section 
 a homomorphism of $\scP$-spaces is called a \textit{weak equivalence} if its underlying map 
 of spaces is a  weak homotopy equivalences, and a homomorphism of $\scO$-algebras $f: \scA\to \scB$
 is called a \textit{weak equivalence} if $Bf$ is a weak equivalence of $B\scO$-spaces.
 Two $\scP$-spaces $X$ and $Y$ and two $\scO$-algebras $\scA$ and $\scB$ are called
\textit{weakly equivalent} if there is a chain of weak equivalences connecting $X$ and $Y$ respectively $\scA$ and $\scB$.
\end{defi}

Let $\scO$ be a $\Sigma$-free operad in $\Cat$. We want to compare the categories $\scO\mbox{-}\Cat$ and $B\scO\mbox{-}\Top$. The classifying
space functor maps an $\scO$-algebra $\scC$ to the $B\scO$-algebra $B\scC$. In \cite{FV} we showed that for each $B\scO$-space $X$ there
is a simplicial $\scO$-algebra $\scA_\ast$ and a sequence of natural weak equivalences of $B\scO$-spaces connecting
$X$ and $|B\scA_\ast|$.
By Theorem \ref{simpl4} there is an $\scO$-algebra $\scC$ such that $B\scC$ and 
$|B\scA_\ast|$ are weakly equivalent $B\scO$-spaces. So after localization with respect to the weak equivalences the categories
$\scO\mbox{-}\Cat$ and $B\scO\mbox{-}\Top$ are equivalent.

Since $B\scO\mbox{-}\Top$ carries a Quillen model structure with the weak equivalences of \ref{appli1} its localization
$B\scO\mbox{-}\Top[\we^{-1}]$ with respect to these weak equivalences exists \cite[Theorem B]{SV2}. We do not know whether or not
$\scO\mbox{-}\Cat$ carries a model structure, but combining our previous results with a result of
Schlichtkrull and Solberg \cite{SS} we obtain:

\begin{theo}\label{appl3} Let $\scO$ be a $\Sigma$-free operad in $\Cat$. Then
 the localization $\scO\mbox{-}\Cat[\we^{-1}]$ exists and the classifying space functor
 induces an equivalence of categories
 $$
 \scO\mbox{-}\Cat[\we^{-1}]\quad\simeq\quad B\scO\mbox{-}\Top[\we^{-1}].$$
\end{theo}
\begin{proof}
 Let $T=R\circ S:\Top\to \Top$ be the standard CW-approximation functor, i.e. the composite of the singular functor $S$ and the topological 
 realization functor, which we denote by $R$ in this proof. 
 Let $X$ be a $B\scO$-space. Then $TX$ is a $B\scO$-space
 and the natural map $TX\to X$ is a weak equivalence of $B\scO$-spaces. To see recall that $B\scO=RN_\ast \scO$.  Consider the diagram
 $$\xymatrix{
 RN_\ast(\scO(n))\times RS(X)^n\ar[rr]^(.48){R\mu N_\ast\scO(n)\times \id}\ar[d]_{\id\times \varepsilon^n} && RSRN_\ast(\scO(n))\times RS(X)^n
 \ar[rr]^(.65){RS\alpha}\ar[dll]^{\varepsilon RN_\ast\scO(n)\times \varepsilon^n} && RS(X)\ar[d]^{\varepsilon} \\
  RN_\ast(\scO(n))\times X^n\ar[rrrr]^{\alpha} &&&& X
  }
  $$
  where $\alpha:RN_\ast(\scO(n))\times X^n\to X$ is a structure map, and $\mu$ and $\varepsilon$ are the unit and counit of the
  adjunction
  $$R:\SSets\leftrightarrows \Top:S.$$
 Since all functors are product preserving the right square commutes by naturality, and the triangle commutes because $\varepsilon R\circ R\mu=\id$.
 
 In \cite[Section 5]{FV} we constructed a functor $\widehat{Q}_\bullet: B\scO\mbox{-}\Top\to \scO\mbox{-}\Cat^{\Delta^{\op}}$ and showed that there is
 a sequence of natural weak equivalences in $\scO\mbox{-}\Cat^{\Delta^{\op}}$ joining $\scC_\bullet$ and $\widehat{Q}_\bullet(B\scC)$, where $\scC$ is
 an $\scO$-algebra and $\scC_\bullet$ is the constant simplicial $\scO$-algebra on $\scC$ \cite[Lemma 6.6]{FV}. Here we call a 
 map $f_\bullet:X_\bullet \to Y_\bullet$ of simplicial $B\scO$-spaces a weak equivalence if its realization $|f_\bullet|:|X_\bullet| \to |Y_\bullet|$
 is a weak equivalence, and a map $g_\bullet: \scA_\bullet  \to \scB_\bullet$ of simplicial $\scO$-algebras a weak equivalence, if
 $B(g_\bullet)$ is a weak equivalence in $B\scO-\Top^{\Delta^{\op}}$. If $X$ is a $B\scO$-space whose underlying space is a CW-complex, we also
 showed that there is a sequence of natural weak equivalences in $B\scO\mbox{-}p\Top^{\Delta^{\op}}$ joining $B\widehat{Q}_\bullet (X)$ 
and the constant simplicial
 $B\scO$-space $X_\bullet$ \cite[Lemma 6.3]{FV}.
 
 We define 
 $$
 F=H_{\cat}\circ \widehat{Q}_\bullet\circ T:B\scO\mbox{-}\Top \to \scO\mbox{-}\Cat.
 $$
 Let $X$ be a $B\scO$-space and $X_\bullet$ the constant simplicial $B\scO$-space on $X$.  By Theorem \ref{simpl4} we have a natural weak equivalence
 $$H_{\top}(TX_\bullet)\to |TX_\bullet|\cong TX\to X.$$
  By Theorem \ref{simpl4} $H_{\top}$ maps weak equivalences in $B\scO\mbox{-}p\Top^{\Delta^{\op}}$ to weak equivalences in $B\scO\mbox{-}\Top$. 
So if we apply $H_{\top}$
  to the second sequence of weak equivalences we obtain a sequence of weak equivalences
 joining $H_{\top}(TX_\bullet)$ and $H_{\top}(B\widehat{Q}_\bullet (TX))$, and, again by Theorem \ref{simpl4}, there is a weak equivalence
 $H_{\top}(B\widehat{Q}_\bullet (TX))\to B(H_{\cat}(\widehat{Q}_\bullet (TX))$. So there is a sequence of natural weak equivalences joining $B\circ F$ and $\Id$.
 
 Let $\scC$ be an $\scO$-algebra. By Proposition \ref{simpl5} there is a weak equivalence $H_{\cat}(\scC_\bullet)\to \scC$, and since $B\scC$ is a CW-complex, the
 natural map $TB\scC\to B\scC$ induces a weak equivalence $\widehat{Q}_\bullet(TB\scC)\to \widehat{Q}_\bullet(B\scC)$. By applying $H_{\cat}$ to the first sequence 
 of weak equivalences above we obtain a sequence of natural weak equivalences joining $H_{\cat}(\scC_\bullet)$ and $H_{\cat}(\widehat{Q}_\bullet(B\scC))$,
 because $H_{\cat}$ maps weak equivalences to weak equivalences. Altogether we obtain a sequence of natural weak equivalences in $\scO\mbox{-}\Cat$ joining
 $F\circ B$ and $\Id$.
 
 Then by \cite[Prop. A.1]{SS}, the existence of the localization $B\scO\mbox{-}\Top[\we^{-1}]$ implies
the existence of the localization $\scO\mbox{-}\Cat[\we^{-1}]$ and the equivalence
 $$
 \scO\mbox{-}\Cat[\we^{-1}]\quad\simeq\quad B\scO\mbox{-}\Top[\we^{-1}].$$

\end{proof}

From Theorem \ref{appl3} we obtain
 the results about iterated loop spaces of \cite[Section 8]{FSV} without referring to the fairly complicated homotopy colimit construction in categories of algebras over $\Sigma$-free
 operads in $\Cat$. We include a short summary of these applications, because we now have statements about 
genuine localizations rather than localizations up to equivalence. For further details, in particular the group completion
functors, see \cite{FSV}.
 
 \begin{leer}\label{appl5} \textbf{Notations:} $\mathcal{B}r$ denotes the operad codifying strict braided monoidal categories, 
 i.e. braided monoidal categories which are strictly associative and have a 
  strict 2-sided unit (recall that any braided monoidal category is equivalent to a strict one).
  
  $\mathcal{M}_n$ denotes the operad codifying $n$-fold monoidal categories, $1\leq n\leq \infty$, introduced in \cite{BFSV}.
  
  $\mathcal{P}erm$ denotes the operad codifying permutative categories.
  
  $\mathcal{C}_n$, $1\leq n\leq \infty$, denotes the little $n$-cubes operad.
  \end{leer}

 \begin{theo}\label{appl6} The composites of the classifying space functors and the change of operads functors induce equivalences of categories
  $$\begin{array}{rcccl}
   \mathcal{M}_n\mbox{-}\Cat[\we^{-1}] & \simeq & B\mathcal{M}_n\mbox{-}\Top[\we^{-1}]  & 
\simeq &  \mathcal{C}_n\mbox{-}\Top[\we^{-1}],\quad 1\leq n\leq \infty\\
   \mathcal{B}r\mbox{-}\Cat[\we^{-1}] & \simeq & B\mathcal{B}r\mbox{-}\Top[\we^{-1}]  & \simeq &  \mathcal{C}_2\mbox{-}\Top[\we^{-1}]\\
   \mathcal{P}erm\mbox{-}\Cat[\we^{-1}] & \simeq & B\mathcal{P}erm\mbox{-}\Top[\we^{-1}]  & \simeq &  \mathcal{C}_\infty\mbox{-}\Top[\we^{-1}]
    \end{array}
$$
 \end{theo}

 It is well known that the group completion of a $\mathcal{C}_n$-space is an $n$-fold loop space for $1\leq n\leq \infty$. 
Let $\Omega^n\mbox{-}\Top$ denote the category
 of $n$-fold loop spaces and $n$-fold loop maps. A \textit{weak equivalence} in  $\Omega^n\mbox{-}\Top$ is an $n$-fold
 loop map whose underlying map is a weak homotopy equivalence, or equivalently if the
May delooping \cite{May1} of the $n$-fold loop map is an equivalence.  Again by \cite[Prop. A.1]{SS}
the localization with respect to these weak equivalences exists. Let $\we_g$ denote the classes
 of morphisms in $\mathcal{B}r\mbox{-}\Cat,\ \mathcal{M}_n\mbox{-}\Cat$, and $\mathcal{P}erm\mbox{-}\Cat$ which are 
mapped to weak equivalences by the composites of the
 classifying space functors, the change of operads functors, and the group completion functors. The localizations with
 respect to these weak equivalences exist by the same argument and we have:
 
 \begin{theo}\label{appl7}
  The composites of the
 classifying space functors, the change of operads functors, and the group completion functors induce equivalences of categories
 $$\begin{array}{rcl}
  \mathcal{M}_n\mbox{-}\Cat[{\we_g}^{-1}] & \simeq &  {\Omega^n\Top}[\we^{-1}], \qquad 1\leq n\leq \infty\\
  \mathcal{B}r\mbox{-}\Cat[{\we_g}^{-1}] & \simeq &{\Omega^2\Top}[\we^{-1}]\\
  \mathcal{P}erm\mbox{-}\Cat[{\we_g}^{-1}] & \simeq & {\Omega^\infty\Top}[\we^{-1}].
   \end{array}
$$
 \end{theo}


%
%
%
%

\end{document}